\documentclass[reqno]{amsart}
\usepackage{amssymb}
\usepackage[usenames, dvipsnames]{color}
\usepackage[hypertex]{hyperref}

\newcommand{\mysection}[1]{\section{#1}
\setcounter{equation}{0}}

\newtheorem{theorem}{Theorem}[section]
\newtheorem{corollary}[theorem]{Corollary}
\newtheorem{lemma}[theorem]{Lemma}

\theoremstyle{definition}
\newtheorem{remark}[theorem]{Remark}

\theoremstyle{definition}

\theoremstyle{definition}
\newtheorem{assumption}[theorem]{Assumption}

\makeatletter
\def\dashint{\operatorname%
{\,\,\text{\bf--}\kern-.98em\DOTSI\intop\ilimits@\!\!}}
\makeatother

\def\bR{\mathbb{R}}
\def\bZ{\mathbb{Z}}

\def\bC{\mathbb{C}}

\def\sba{\textsl{\textbf{a}}}
\def\sbx{\textsl{\textbf{x}}}
\def\sby{\textsl{\textbf{y}}}
\def\sbz{\textsl{\textbf{z}}}
\def\sbw{\textsl{\textbf{w}}}

\def\fL{\mathfrak{L}}

\def\cA{\mathcal{A}}
\def\cB{\mathcal{B}}
\def\cC{\mathcal{C}}

\def\cT{\mathcal{T}}
\def\cL{\mathcal{L}}

\def\cI{\mathcal{I}}

\newcommand{\Div}{\operatorname{div}}
\newcommand{\dist}{\text{dist}}
\begin{document}
\title[]{Elliptic equations in divergence form with partially  BMO coefficients}

\author[H. Dong]{Hongjie Dong}
\address[H. Dong]{Division of Applied Mathematics, Brown University,
182 George Street, Providence, RI 02912, USA}
\email{Hongjie\_Dong@brown.edu}
\thanks{H. Dong was partially supported by a start-up funding from the Division of Applied Mathematics of Brown University, NSF grant number DMS-0635607 from IAS, and NSF grant number DMS-0800129.}

\author[D. Kim]{Doyoon Kim}
\address[D. Kim]{Department of Mathematics, University of Southern California,
3620 South Vermont Avenue, KAP 108, Los Angeles, CA 90089-2532, USA}
\email{doyoonki@usc.edu}



\begin{abstract}
The solvability in Sobolev spaces is proved for divergence form second order elliptic equations in the whole space, a half space, and a bounded Lipschitz domain. For equations in the whole space or a half space, the leading coefficients $a^{ij}$ are assumed to be measurable in one direction and have small BMO semi-norms  in  the other directions. For equations in a bounded domain, additionally we  assume that $a^{ij}$ have small BMO semi-norms in a neighborhood of the boundary of  the domain.
We give a unified approach of both the Dirichlet boundary problem and the conormal derivative problem.
We also investigate elliptic equations in Sobolev spaces with mixed norms under the same assumptions on the coefficients.
\end{abstract}

\maketitle

\mysection{Introduction}

We study the solvability of elliptic operators in divergence form
\begin{equation}							 \label{eq0617_02}
\cL u = (a^{ij} u_{x^i}+ a^{j}u)_{x^j} +b^{i} u_{x^i} + c u
\end{equation}
in Sobolev spaces with rough leading coefficients. Throughout the paper,
the usual summation conventions over repeated indices are enforced. We assume all the coefficients are bounded and measurable, and $a^{ij}$ are uniformly elliptic.

There have been many research activities in this direction. For divergence form elliptic equations the strongest results up to date  can be found in Byun \cite{Byun05a}, Byun and Wang \cite{ByunWang04}, \cite{ByunWang05}, and Krylov \cite{Krylov_2005}.

In \cite{Byun05a}, the $W^{1}_p$ solvability was obtained for the Dirichlet problem of divergence form elliptic equations in a Lipschitz domain with a small Lipschitz constant. For equations in a so-called Reifenberg flat domain,
the solvability of the Dirichlet problem and the conormal derivative problem was established in \cite{ByunWang04} and \cite{ByunWang05}. In those papers the coefficients $a^{ij}$ are assumed to have small BMO semi-norms and lower order terms are not included. The main tools in \cite{Byun05a}, \cite{ByunWang04}, \cite{ByunWang05} are the weak compactness, the Hardy-Littlewood maximal function, and the Vitali
covering lemma originally used by M. Safonov. Before that, the solvability for the Dirichlet and Neumann problems of divergence form elliptic equations with VMO coefficients were obtained in \cite{DFG} for  $C^{1,1}$ domains, and in \cite{AuscherQafsaoui} for $C^{1}$ domains.

In \cite{Krylov_2005}, Krylov gave a unified approach of the $L_p$ solvability of both divergence and non-divergence form parabolic and elliptic equations with leading coefficients VMO in the spatial variables (and measurable in the time variable in the parabolic case). Unlike the arguments in \cite{CFL2}, \cite{DFG} and \cite{HHH}, which are based on certain estimates of Calder\'on-Zygmund theorem and the Coifman-Rochberg-Weiss commutator
theorem, the proofs in \cite{Krylov_2005} rely mainly on pointwise estimates of sharp functions of spatial derivatives of solutions. It is worth noting that although the results in \cite{Krylov_2005} are stated for equations with VMO coefficients, the proofs there only require $a^{ij}$ to have locally small BMO semi-norms.
We also remark that for divergence form parabolic equations a similar result was also obtained in Byun \cite{Byun07} by adapting the approach in \cite{Byun05a}. Krylov's method was later improved and generalized in \cite{DK08}, \cite{KimKrylov07}-\cite{Kim07a}, \cite{Krylov_2007_mixed_VMO} and \cite{Krylov08}. With the leading coefficients in the same class, Krylov \cite{Krylov_2007_mixed_VMO} established the solvability of both divergence and non-divergence parabolic equations in mixed-norm Sobolev spaces.

There are many other results  in the literature regarding the $L_p$ theory of second order parabolic and elliptic equations with discontinuous coefficients. For non-divergence form equations, we refer the reader to \cite{MR1239929}, \cite{CFL2}, \cite{Krylov:book:2008}, \cite{Lo72}, \cite{Weid02} and references therein. For divergence form equations, see also \cite{Shen05} and references therein.


The theory of elliptic and parabolic equations with partially VMO coefficients
is originated in Kim and Krylov \cite{KimKrylov07}, where the authors proved the $W^2_p$ solvability of elliptic equations in non-divergence form with leading coefficients measurable in a fixed direction and VMO in  the others. Very recently, their result was generalized by Krylov \cite{Krylov08}, where the leading coefficients are assumed to be measurable in
one direction and VMO in the orthogonal directions in each small
ball with the direction depending on the ball. For non-divergence parabolic equations, the $W^{1,2}_{q,p}$ solvability was established in Kim \cite{Kim07a}, in which most leading coefficients are measurable in time variable as well as one spatial variable, and VMO in the other variables. We remark that to our best knowledge, at the time of this writing, all known results concerning $L_p$ solvability of elliptic and parabolic equations with partially VMO/BMO coefficients are only for non-divergence form.

In this paper we consider divergence form elliptic equations in the whole space, a half space and a Lipschitz domain with a small Lipschitz constant. We deal with equations with partially BMO leading coefficients with locally small BMO semi-norms (Theorem \ref{th081901}), a class of coefficients which is more general than those treated  previously in \cite{Krylov_2005}, \cite{Byun05a}, \cite{ByunWang04} and \cite{ByunWang05}. More precisely, we assume the coefficients $a^{ij}$ are measurable in $x^1$ direction and BMO in the other directions with locally small BMO semi-norms (see Assumption \ref{assumption20080424} for a more rigorous definition). This is the same class of coefficients considered in \cite{KimKrylov07}, in which non-divergence form elliptic equations are studied. For equations in a Lipschitz domain, additionally we  assume that $a^{ij}$ have small BMO semi-norms in a neighborhood of the boundary of the domain.
Under these assumptions, we establish the unique $W^1_p$ solvability of divergence form elliptic equations. We give a unified approach of both the Dirichlet boundary problem and the conormal derivative problem in a half space (Theorem \ref{thm6.6}, \ref{thm6.7}) and in a bounded Lipschitz domain (Theorem \ref{thmA} and \ref{thmB}). We also investigate elliptic equations in Sobolev spaces with mixed norms under the same assumption on the coefficients. We point out that, as in \cite{Krylov_2005} and \cite{Krylov_2007_mixed_VMO}, one feature of these results
is that the matrix $\{a^{ij}\}$ is not assumed to be symmetric.

One of the motivations of the paper is the following problem. Consider the equation $\left(a^{ij}u_{x^i}\right)_{x^j} = \Div g$ in $B_2$, the ball of radius $2$ centered at the origin, with zero Dirichlet boundary condition. The coefficients $a^{ij}$ are assumed to be bounded, uniformly elliptic and piecewise uniformly continuous on $B_1$ and $B_2\setminus B_1$. This is a very natural problem and the $W^1_2$ solvability of it follows immediately from the Lax-Milgram lemma. However, for the $W^1_p$ solvability when $p\neq 2$, it seems to us that none of the results above are applicable in this case. We will give a solution to the problem at the end of Section \ref{sec082001} as an application of our main results.

Our approach is based on the aforementioned method from \cite{Krylov_2005}. However, since $a^{ij}$ are merely measurable in $x^1$, we are only able to estimate the sharp function of $u_{x'}$, not the full gradient $u_x$ as in \cite{Krylov_2005}.
Here and throughout the paper, we denote $x' = (x^2,\cdots,x^d) \in \bR^{d-1}$,
so by $u_{x'}$ we mean one of $u_{x^i}$, $i = 2, \cdots, d$,
or the whole collection of them.
Roughly speaking, the main difficulty is to bound $u_{x^1}$ by $u_{x'}$. One idea in the paper is to break the `symmetry' of the coordinates so that $x^1$ is distinguished from $x'$.
Another idea is to estimate the sharp of $a^{11}u_{x^1}$ instead of $u_{x^1}$. This estimate together with a generalized Stein-Fefferman theorem proved in \cite{Krylov08} enables us to bound $u_{x^1}$. The main advantage of the approach is that here we can obtain the boundary estimate
immediately from the estimate in the whole space since the leading coefficients are allowed to be just measurable in one direction. In a forthcoming paper, we will extend our results to systems with variably partially BMO coefficients.

A brief outline of the paper:  in the next section, we introduce the notation and state the main results, Theorem \ref{th081901}, \ref{thm6.6}, \ref{thm6.7}, \ref{thmA} and \ref{thmB}. Section \ref{sec_aux} is devoted to several auxiliary results which will be used later, in which we estimate the $L_p$ norm $u_{x^1}$ by the $L_p$ norm of $u_{x'}$ (Theorem \ref{theorem5.3}). Then in Section \ref{ellSec1}, we give an estimate of the sharp function of $u_{x'}$. By combining this with Theorem \ref{theorem5.3}, we are able to prove Theorem \ref{th081901} in Section \ref{ellSec2}. Theorem \ref{thm6.6} and \ref{thm6.7} are proved in Section \ref{sec6}, while Theorem \ref{thmA} and \ref{thmB} are proved in Section \ref{sec7},  Finally, the last four sections are devoted to the mixed norm estimate.

\mysection{Main results}								 \label{sec082001}

Before we state our assumptions and main theorems, we introduce some necessary notations.
By $\bR^d$ we mean a $d$-dimensional Euclidean space and
a point in $\bR^d$ is denoted by $x=(x^1,\cdots,x^d)=(x^1,x')$.
For given two positive integers $d_1$ and $d_2$ such that
$d_1 + d_2 = d$,
we set
$$
\sbx_1 = (x^1, \cdots, x^{d_1}) \in \bR^{d_1},
\quad
\sbx_2 = (x^{d_1+1}, \cdots, x^d) \in \bR^{d_2}.
$$
That is, for example, $\sbx_1$ represents the first $d_1$ coordinates of $x \in \bR^d$.

If $\Omega$ is an open subset in $\bR^d$,
we define
$$
\| u \|_{L_{q,p}(\Omega)} := \| u \|_{L_q^{\sbx_2}L_p^{\sbx_1}(\Omega)}
= \left(\int_{\bR^{d_2}} \left( \int_{\bR^{d_1}} |u(x)|^p I_{\Omega}(x) \, d \sbx_1 \right)^{q/p} \, \sbx_2 \right)^{1/q}.
$$
Note that, in case $p = q$,  $L_p(\Omega) = L_{p,p}(\Omega) = L_p^{\sbx_2}L_p^{\sbx_1}(\Omega)$.
Set
$$
\bR^{d}_{+} = \{ x \in \bR^d: x = (x^1, \cdots, x^d), x^1 > 0 \}.
$$
A function $u$ belongs to $W_{q,p}^1(\Omega)$ if $u, u_x \in L_{q,p}(\Omega)$.
Unless specified otherwise,
by $L_p$ we mean $L_p(\bR^d)$.
Similarly, whenever we use $L_{q,p}$, $W_p^1$, $L_{p, \text{loc}}$, $W_{p,\text{loc}}$,
and $C_0^{\infty}$,
we understand that $\bR^d$ is omitted.

For a function $f$ in $\bR^{d}$, we set
\begin{equation*}
(f)_{\Omega} = \frac{1}{|\Omega|} \int_{\Omega} f(x) \, dx
= \dashint_{\Omega} f(x) \, dx,
\end{equation*}
where $|\Omega|$ is the
$d$-dimensional Lebesgue measure of $\Omega$.

Throughout the paper we assume that the coefficients $a^{ij}$, $a^{i}$, $b^i$, and $c$ are bounded by a constant $K \ge 1$.
Moreover, we assume the uniform ellipticity condition on $a^{ij}$, i.e.,
$$
\delta |\xi|^2 \le a^{ij}(x)\xi^i \xi^j
$$
for all $x$ and $\xi \in \bR^d$, where $\delta \in (0,1)$.

We need a very mild regularity assumption on the coefficients $a^{ij}$.
To present this assumption, let
$$
B_r(x) = \{ y \in \bR^d: |x-y| < r\},
\quad
B'_r(x') = \{ y' \in \bR^{d-1}: |x'-y'| < r\},
$$
$$
\Gamma_r(x) = (x^1-r, x^1+r) \times B'_r(x').
$$
Set $B_r = B_r(0)$, $B'_r = B'_r(0)$, and
$|B'_r|$ is the $d-1$-dimensional volume of $B'_r(0)$.
Denote
$$
\text{osc}_{x'}\left(a^{ij},\Gamma_r(x)\right)
= \frac{1}{2r}\int_{x^1-r}^{x^1+r}
\dashint_{B'_r(x')} \big| a^{ij}(y^1, y') - \dashint_{B'_r(x')} a^{ij}(y^1,z') \, dz' \big| \, dy' \, dy^1,
$$
where
$$
\dashint_{B'_r(x')} a^{ij}(y^1,z') \, dz'
= \frac{1}{|B'_r|} \int_{B'_r(x')} a^{ij}(y^1,z') \, dz'.
$$
Then we set
$$
a^{\#}_R = \sup_{x \in \bR^d} \sup_{r \le R}  \sup_{ij} \text{osc}_{x'} \left(a^{ij},\Gamma_r(x)\right).
$$

The following assumption  contains a parameter $\gamma>0$, which will be
specified later.
\begin{assumption}[$\gamma$]                          \label{assumption20080424}
There is a constant $R_0\in (0,1]$ such that $a_{R_0}^{\#} \le \gamma$.
\end{assumption}

We state the main results concerning elliptic equations in divergence form in the usual Sobolev spaces $W_p^1$.
For equations in Sobolev spaces with {\em mixed norms} $W_{q,p}^1$, as indicated in the introduction, our results are presented in Section \ref{sec9}.

\begin{theorem}[Equations in the whole space]
    							\label{th081901}
Let $p \in (1,\infty)$
and $f$, $g = (g_1, \cdots, g_d) \in L_p$.
Then there exists a constant $\gamma=\gamma(d,p,\delta,K)$
such that, under Assumption \ref{assumption20080424} ($\gamma$),
the following hold true.

\noindent
(i)
For any $u \in W_p^1$ satisfying
\begin{equation}							 \label{eq081902}
\cL u - \lambda u = \Div g + f,
\end{equation}
we have
\begin{equation}							 \label{eq080904}
\lambda \| u \|_{L_p} + \sqrt{\lambda} \| u_x \|_{L_p}
\le N \sqrt{\lambda} \| g \|_{L_p} + N\| f \|_{L_p},
\end{equation}
provided that $\lambda \ge \lambda_0$,
where $N$ and $\lambda_0 \ge 0$
depend only on $d$, $p$, $\delta$, $K$ and $R_0$.

\noindent
(ii)
For any  $\lambda > \lambda_0$, there exists a unique $u \in W_p^1$ satisfying \eqref{eq081902}.

\noindent
(iii)
If $a^i = b^i = c = 0$ and $a^{ij} = a^{ij}(x^1)$, i.e., measurable functions of $x^1 \in \bR$  only with no regularity assumptions,
then one can take $\lambda_0 = 0$.
\end{theorem}

The next two theorems are about the Dirichlet problem and the conormal derivative problem
on a half space.

\begin{theorem}[Dirichlet problem on a half space]
                                \label{thm6.6}
Let $p \in (1,\infty)$ and $f$, $g = (g_1, \cdots, g_d) \in L_p(\bR^{d}_{+})$.
Then there exists a constant $\gamma = \gamma(d,p,\delta,K)$
such that,
under Assumption \ref{assumption20080424} ($\gamma$),
for any  $u \in W^1_p(\bR^d_{+})$ satisfying
\begin{equation}	\label{eq2008061901}
\left\{
  \begin{array}{ll}
    \cL u-\lambda u=\Div g+f & \hbox{in $\bR^d_+$} \\
    u=0 & \hbox{on $\partial \bR^d_+$}
  \end{array},
\right.
\end{equation}
we have
\begin{equation}
                                \label{eq24.2.52pm}
\sqrt{\lambda}\|u_x\|_{L_p(\bR^{d}_{+})}
+\lambda \|u\|_{L_p(\bR^{d}_{+})}
\leq N\sqrt{\lambda}\|g\|_{L_p(\bR^{d}_{+})}
+N\|f\|_{L_p(\bR^{d}_{+})},
\end{equation}
provided that $\lambda \ge \lambda_0$,
where $N$ and $\lambda_0 \ge 0$ depend only on $d$, $p$, $\delta$, $K$, and $R_0$.
Moreover, for any $\lambda > \lambda_0$ and $g,f\in L_p(\bR^{d}_{+})$, there exists a unique $u\in W^1_p(\bR^{d}_{+})$ satisfying \eqref{eq2008061901}.
\end{theorem}

\begin{theorem}[Conormal derivative problem on a half space]
                                \label{thm6.7}
Let $p \in (1,\infty)$ and $f$, $g = (g_1, \cdots, g_d) \in L_p(\bR^{d}_{+})$.
Then there exists a constant $\gamma = \gamma(d, p, \delta, K)$
such that,
under Assumption \ref{assumption20080424} ($\gamma$),
for any $u \in W^1_p(\bR^d_{+})$ satisfying
\begin{equation}	\label{eq2008061902}
\left\{
  \begin{array}{ll}
    \cL u-\lambda u=\Div g+f & \hbox{in $\bR^d_+$} \\
    a^{i1}u_{x^i}+a^1u=g_1 & \hbox{on $\partial \bR^d_+$}
  \end{array},
\right.
\end{equation}
we have
\begin{equation}
                                        \label{eq27.3.49pm}
\sqrt{\lambda}\|u_x\|_{L_p(\bR^{d}_{+})}
+\lambda \|u\|_{L_p(\bR^{d}_{+})}
\leq N\sqrt{\lambda}\|g\|_{L_p(\bR^{d}_{+})}
+N\|f\|_{L_p(\bR^{d}_{+})},
\end{equation}
provided that $\lambda \ge \lambda_0$,
where $N$ and $\lambda_0$ depend only on $d$, $p$, $\delta$, $K$, and $R_0$.
Moreover, for any $\lambda > \lambda_0$ and $g,f\in L_p(\bR^{d}_{+})$, there exists a unique $u\in W^1_p(\bR^{d}_{+})$ satisfying \eqref{eq2008061902}.
\end{theorem}

Solutions of \eqref{eq2008061902} are understood in the weak sense. More precisely, we say $u\in W^1_p(\bR_+^d)$ satisfies \eqref{eq2008061902} if we have
\begin{equation}
                                                \label{eq27.3.58pm}
\int_{\bR_+^d}\left(-a^{ij}u_{x^i}\phi_{x^j}-a^ju \phi_{x^j}+b^i u_{x^i}\phi+(c-\lambda)u\phi\right)\,dx=\int_{\bR_+^d}\left(-g_j\phi_{x^j}+f\phi\right)\,dx
\end{equation}
for any $\phi\in W^1_{p'}(\bR^d_+)$, where $p'$ satisfy $1/p+1/p'=1$. For discussions about the conormal derivative problem, we refer the reader to \cite{Lieb1} and \cite{Lieb2}.

Next we consider the solvability of divergence form elliptic equations in domains with the homogeneous Dirichlet boundary condition:
\begin{equation}	\label{eq27.8.51}
\left\{
  \begin{array}{ll}
    \cL u=\Div g+f & \hbox{in $\Omega$} \\
    u=0 & \hbox{on $\partial \Omega$}
  \end{array}.
\right.
\end{equation}
We shall impose a little bit more regularity assumption on $a^{ij}$ near the boundary. For any $x\in \bR^d$, denote
$$
\dist(x,\partial \Omega)=\inf_{y\in \partial\Omega}|x-y|.
$$
\begin{assumption}[$\gamma$]
                                        \label{assump1}
There is a constant $R_1\in (0,1]$ such that, for any $x_0\in \bR^d$ with $\dist(x_0,\partial\Omega)\le R_1$ and any $r\in (0,R_1]$, we have
$$
\sup_{ij}\dashint_{B_r(x_0)}|a^{ij}(x)-(a^{ij})_{B_r(x_0)}|\,dx\le \gamma.
$$
\end{assumption}

We also impose the same assumption on domains as in \cite{Byun05a}, i.e. the boundary $\partial \Omega$ of the domain $\Omega$ is locally the graph of a Lipschitz continuous function with a small Lipschitz constant. More precisely, we make the following assumption containing a parameter $\theta\in (0,1]$, which will be specified later.

\begin{assumption}[$\theta$]
                                    \label{assump2}
There is a constant $R_2\in (0,1]$ such that, for any $x_0\in \partial\Omega$ and $r\in(0,R_2]$, there exists a Lipschitz
function $\phi$: $\bR^{d-1}\to \bR$ such that
$$
\Omega\cap B_r(x_0) = \{x \in B_r(x_0)\, :\, x^1 >\phi(x')\}
$$
and
$$
\sup_{x',y'\in B_r'(x_0'),x' \neq y'}\frac {|\phi(y')-\phi(x')|}{|y'-x'|}\le \theta
$$
in some coordinate system.
\end{assumption}
Note that all $C^1$ domains satisfy this assumption for any $\theta>0$.

\begin{theorem}[Dirichlet problem on a bounded domain]
                                            \label{thmA}
Let $p\in (1,\infty)$ and $\Omega$ be a bounded domain. Assume $a^i_{x^i}+c\le 0$ in $\Omega$ in the weak sense. Then there exist $\gamma=\gamma(d,p,\delta,K)$ and $\theta=\theta(d,p,\delta,K)$ such that, under Assumption \ref{assumption20080424} ($\gamma$), \ref{assump1} ($\gamma$), and Assumption \ref{assump2} ($\theta$), for any $f$, $g = (g_1, \cdots, g_d) \in L_{p}(\Omega)$ there exists a unique $u\in W^1_p(\Omega)$ satisfying \eqref{eq27.8.51}.
Moreover, we have
\begin{equation}
                            \label{eq27.8.52}
\|u\|_{W^1_p(\Omega)}\le N\|f\|_{L_p(\Omega)}+N\|g\|_{L_p(\Omega)},
\end{equation}
where $N$ is independent of $f,g$ and $u$.
\end{theorem}

Our last result is about the solvability of divergence form elliptic equations in domains with the conormal derivative boundary condition:
\begin{equation}	\label{eq27.9.21}
\left\{
  \begin{array}{ll}
    \cL u=\Div g+f & \hbox{in $\Omega$} \\
    a^{ij}u_{x^i}n^j+a^j u n^j=g_j n^j & \hbox{on $\partial \Omega$}
  \end{array},
\right.
\end{equation}
where $n=(n^1,\cdots,n^d)$ is the outward normal direction of $\partial\Omega$, which is defined almost everywhere on $\partial \Omega$. Like before, solutions of \eqref{eq27.9.21} are understood in the weak sense. More precisely, we say $u\in W^1_{p}(\Omega)$ satisfies \eqref{eq27.9.21} if we have
\begin{equation}
                                                \label{eq27.9.28pm}
\int_{\Omega}\left(-a^{ij}u_{x^i}\phi_{x^j}-a^ju \phi_{x^j}+b^i u_{x^i}\phi+ cu\phi \right)\,dx
=\int_{\Omega}\left(-g_j\phi_{x^j}+f\phi\right)\,dx,
\end{equation}
for any $\phi\in W^1_{p'}(\Omega)$.

\begin{theorem}[Conormal derivative problem on a bounded domain]
                                        \label{thmB}
Let $p\in (1,\infty)$ and $\Omega$ be a bounded domain. Assume $a^i_{x^i}+c\le 0$ in $\Omega$ in the weak sense. Then there exist $\gamma=\gamma(d,p,\delta,K)$ and $\theta=\theta(d,p,\delta,K)$ such that, under Assumption \ref{assumption20080424} ($\gamma$), \ref{assump1} ($\gamma$), and Assumption \ref{assump2} ($\theta$),

\noindent
(i) If in the weak sense $a^i_{x^i}+c\equiv 0$ in $\Omega$ and $a^i n^i=0$ on $\partial \Omega$, then for any $f$, $g = (g_1, \cdots, g_d) \in L_{p}(\Omega)$, the equation \eqref{eq27.9.21} has a unique up to a constant solution $u\in W^1_p(\Omega)$ provided that $b^i=c=0$ and $\int_\Omega f\,dx=0$.
Moreover, we have
\begin{equation*}
\|u_x\|_{L_p(\Omega)}\le N\|f\|_{L_p(\Omega)}+N\|g\|_{L_p(\Omega)}.
\end{equation*}

\noindent
(ii) Otherwise, the solution is unique and we have
\begin{equation*}
\|u\|_{W^1_p(\Omega)}\le N\|f\|_{L_p(\Omega)}+N\|g\|_{L_p(\Omega)}.
\end{equation*}
The constant $N$ is independent of $f,g$ and $u$.
\end{theorem}

Here, by $a^i_{x^i}+c\le 0$ in $\Omega$, we mean
$$
\int_{\Omega} (-a^i \phi_{x^i}+c\phi)\,dx\le 0
$$
for any nonnegative $\phi\in C_0^1(\Omega)$. By $a^i_{x^i}+c\equiv 0$ in $\Omega$ and $a^i n^i=0$ on $\partial \Omega$, we mean
\begin{equation}
                                \label{eq0831}
\int_{\Omega} (-a^i \phi_{x^i}+c\phi)\,dx=0.
\end{equation}
for any $\phi\in W^1_2(\Omega)$.

Restricted to equations in Lipschitz domains, Theorem \ref{thmA} and \ref{thmB} improve the previous results in \cite{Byun05a}, \cite{ByunWang04} and \cite{ByunWang05} in two aspects: first we only assume that the leading coefficients have partially small BMO semi-norms in the interior of the domain; second we also allow lower order terms. At the time of this writing it is not clear to us whether our method can be extended to deal with equations in Reifenberg flat domains.
We remark that for the Poisson equation in arbitrary Lipschitz domains but with a restricted range of $p$, the solvability result was established by Jerison and Kenig \cite{JeKe} (see also \cite{Shen05} for a generalization to equations with VMO coefficients).

We end this section by giving an example dealing with elliptic equations with piecewise continuous leading coefficients on a bounded domain.
This is another nice application of Theorem \ref{th081901}, showing the possibility that the results in this paper can be applied to many different equations with {\em not necessarily continuous coefficients}.
For simplicity, consider
\begin{equation}							 \label{eq092501}
\left(a^{ij}u_{x^i}\right)_{x^j} = \Div g
\quad \text{in} \,\, B_2,
\quad
u|_{\partial B_2} = 0,
\end{equation}
where each $a^{ij}$ is piecewise uniformly continuous on $B_1$ and $B_2 \setminus B_1$.
As always, $a^{ij}$ are assumed to be uniformly elliptic.
For the solvability of the equation \eqref{eq092501} in $W_p^1(B_2)$,
Theorem \ref{thmA} is not applicable because the coefficients $a^{ij}$ do not have partially small BMO semi-norms in any fixed directions.
However, upon having an appropriate partition of unity and change of variables,
the interior estimate is derived from the $L_p$-estimate of equations with piecewise continuous coefficients.
Here by `piecewise continuous coefficients' we mean coefficients $a^{ij}$ continuous on $\bR^d_+$ and on $\bR^d \setminus \bR^d_+$.
Needless to say, this class of coefficients satisfies the assumptions of Theorem \ref{th081901}. The interior and boundary estimates give us
$$
\|u_x\|_{L_p(B_2)}\le N\|g\|_{L_p(B_2)}+N\|u\|_{L_p(B_2)}.
$$
Then, for $p>2$, one can use the argument in the proof of Theorem \ref{thmA} below to absorb the term $N\|u\|_{L_p(B_2)}$ to the left-hand side. Thus we obtain an estimate as in Theorem \ref{thmA}. The estimate when $p\in (1,2)$ follows from the duality argument.
Consequently, for a given $g \in L_p(B_2)$, $1 < p < \infty$, there exists a unique $u \in W_p^1(B_2)$
satisfying \eqref{eq092501}.

\mysection{Auxiliary results for equations with measurable coefficients}
\label{sec_aux}

In this section
we set
$$
\cL_0 u = \left( a^{ij} u_{x^i} \right)_{x^j},
$$
and we do not impose any regularity assumptions on the
coefficients of the operator $\cL_0$, except $a^{11}$.
The coefficient $a^{11}$ is assumed to be a measurable function of $x^1$ only or satisfying
\begin{assumption}[$\gamma$]                          \label{assumption20080424a}
There is a constant $R_0\in (0,1]$ such that $a_{R_0}^{11,\#} \le \gamma$.
\end{assumption}
Here $\gamma>0$ is a constant to be specified, and
$$
a^{11,\#}_R = \sup_{x \in \bR^d} \sup_{r \le R}  \text{osc}_{x'} \left(a^{11},\Gamma_r(x)\right).
$$
However, in Theorem \ref{theorem08061901} all coefficients including $a^{11}$ are measurable functions of $x \in \bR^d$ with no regularity assumptions.

The first result is the classical $L_2$-estimate for elliptic operators in divergence form with measurable coefficients.

\begin{theorem}			\label{theorem08061901}
There exists $N = N(d, \delta)$
such that, for any $\lambda \ge 0$,
$$
\sqrt{\lambda}\| u_x \|_{L_2} + \lambda \| u\|_{L_2} \le N \left( \sqrt{\lambda}\| g \|_{L_2} + \| f \|_{L_2} \right),
$$
provided that $u \in W_2^1$, $f$, $g = (g_1, \cdots, g_d) \in L_2$,
and
\begin{equation}							 \label{eq080501}
\cL_0 u - \lambda u = \Div g  + f.
\end{equation}
Furthermore, for any $\lambda > 0$ and $f$, $g\in L_2$, there exists a unique solution $u\in W^1_2$ to
the equation \eqref{eq080501}.
\end{theorem}

\begin{proof}
We present a proof for the sake of completeness. Due to the method of continuity it is enough to prove the estimate.
Moreover, by the denseness of $C_0^{\infty}$ in $W_2^1$ it suffices to consider $u \in C_0^{\infty}$.
Then from the equation and the uniform ellipticity condition it follows that
$$
\delta \int_{\bR^d} |u_{x}|^2 \, dx + \lambda \int_{\bR^d} |u|^2 \, dx \le \int_{\bR^d} a^{ij} u_{x^i} u_{x^j} \, dx
+ \lambda \int_{\bR^d} |u|^2 \, dx
$$
$$
= \int_{\bR^d} g_i u_{x^i} \, dx - \int_{\bR^d} f u \, dx
$$
$$
\le \delta/2 \int_{\bR^d} |u_{x^i}|^2 \, dx + N \int_{\bR^d} |g|^2 \, dx
+ \lambda/2 \int_{\bR^d} |u|^2 \, dx + \frac{N}{\lambda} \int_{\bR^d} |f|^2 \, dx,
$$
where $N = N(d, \delta)$. This finishes the proof.
\end{proof}

If the above operator $\cL_0$ is replaced by the Laplace operator $\Delta$,
it is well known that the result as in Theorem \ref{theorem08061901}
holds true not only for $p = 2$ but also for $p \in (1,  \infty)$.
More precisely,
if $\lambda > 0$ and $f$, $g \in L_p$,
then there exists a unique solution $u \in W_p^1$ to the equation
$\Delta u - \lambda u = \Div g + f$.
As above, we have
$$
\| u_x \|_{L_p} + \sqrt{\lambda} \| u\|_{L_p} \le N \left( \| g \|_{L_p} + \lambda^{-1/2}\| f \|_{L_p} \right)
$$
for all $\lambda > 0$.
Using this result, we prove the following theorem.

\begin{theorem}							\label{th080601}
Let $p \in (1,\infty)$, $\lambda > 0$, $\kappa \ge 4$, and $r > 0$.
Assume that $u \in W_{p, \text{loc}}^1$, $f$, $g = (g_1, \cdots, g_d) \in L_{p,\text{loc}}$,
and $\Delta u - \lambda u = \Div g + f$ in $B_{\kappa r}$.
Then there exists a constant $N = N(d, p)$ such that
$$
\dashint_{B_r} | u_x - \left( u_x \right)_{B_r} |^p \, dx
\le N \kappa^{-p} \left( |u_x|^p + \lambda^{p/2}|u|^p\right)_{B_{\kappa r}}
+ N \kappa^{d} \left( |g|^p + \lambda^{-p/2}|f|^p\right)_{B_{\kappa r}}.
$$
\end{theorem}

\begin{proof}
We follow the idea in the proof of Theorem 7.1 in \cite{Krylov_2007_mixed_VMO} taking into account the presence of $\lambda$.
We can certainly assume that $u$, $f$, and $g$ have compact supports.
In addition, we assume that $u$, $f$, and $g$ are infinitely differentiable.
Indeed, if not, we take the standard mollifications and prove the estimate for the mollifications.
Then we take the limit because the concerned constants are independent of the smoothness
of the functions involved.

Take a $\zeta \in C_0^{\infty}$ such that
$$
\zeta = 1 \quad \text{on} \quad B_{\kappa r/2},
\quad
\zeta = 0 \quad \text{on}
\quad
\bR^d \setminus B_{\kappa r}.
$$
Then we find a unique solution $w \in W_p^1$
to the equation
$$
\left( \Delta - \lambda \right) w = \Div ( \zeta g ) + \zeta f.
$$
Set $v := u - w$ and observe that
$$
\left( \Delta - \lambda \right) v = \Div ( (1-\zeta) g ) + (1-\zeta) f.
$$
The classical theory on elliptic equations in divergence form indicates that $w$ and $v$ are infinitely differentiable.
In addition, in $B_{\kappa r/2}$,
$$
\left( \Delta - \lambda \right) v = 0.
$$
Then if we view $v$ as a function in $C_{\text{loc}}^{\infty}(\bR^{d+1})$ independent of $t$, by Lemma 7.4 in \cite{Krylov_2007_mixed_VMO}
\begin{equation}							 \label{eq081502}
\left(|v_{x}-\left(v_{x} \right)_{B_r}|^p\right)_{B_r} \le N \kappa^{-p}\left(|v_x|^p+\lambda^{p/2} |v|^p\right)_{B_{\kappa r}},
\end{equation}
where $N$ depends only on $d$ and $p$.

On the other hand, we have
$$
\| w_x \|_{L_p}
+ \sqrt{\lambda}  \| w \|_{L_p}
\le N \left(\| \zeta g \|_{L_p}
+ \lambda^{-1/2} \| \zeta f \|_{L_p} \right),
$$
which implies
$$
\left(|w_x|^p\right)_{B_r}
\le N r^{-d} \left(\| \zeta g \|^p_{L_p}
+ \lambda^{-p/2} \| \zeta f \|^p_{L_p} \right)
\le N \kappa^{d} \left( |g|^p + \lambda^{-p/2} |f|^p \right)_{B_{\kappa r}},
$$
$$
\left(|w_x|^p\right)_{B_{\kappa r}}
+ \lambda^{p/2}  \left(|w|^p\right)_{B_{\kappa r}}
\le N \left(|g|^p + \lambda^{-p/2}|f|^p\right)_{B_{\kappa r}}.
$$
From these inequalities as well as \eqref{eq081502},
we see that
$$
\dashint_{B_r} | u_x - \left( u_x \right)_{B_r} |^p \, dx
\le N\left(|v_{x}-(v_{x})_{B_r}|^p\right)_{B_r}
+ N\left(|w_{x}|^p\right)_{B_r}
$$
$$
\le N \kappa^{-p}\left(|v_x|^p+\lambda^{p/2} |v|^p\right)_{B_{\kappa r}}
+ N \kappa^{d} \left( |g|^p + \lambda^{-p/2}|f|^p\right)_{B_{\kappa r}}.
$$
We also have
$$
\left(|v_x|^p+\lambda^{p/2} |v|^p\right)_{B_{\kappa r}}
\le N \left(|u_x|^p+\lambda^{p/2} |u|^p\right)_{B_{\kappa r}}
+ N \left(|{w}_x|^p+\lambda^{p/2} |w|^p\right)_{B_{\kappa r}}
$$
$$
\le N \left(|u_x|^p+\lambda^{p/2} |u|^p\right)_{B_{\kappa r}}
+ N \left( |g|^p + \lambda^{-p/2}|f|^p \right)_{B_{\kappa r}}.
$$
Combining the above two sets of inequalities we arrive at the desired inequality in the theorem.
\end{proof}

We frequently make use of the following change of variables to `break' the symmetry of coordinates.
Let
$$
\cL_0 u  - \lambda u = \Div g + f
$$
in $\bR^d$.
For a number $\mu \ge 1$, we set
\begin{equation}							 \label{eq081401}
\bar{a}^{ij}(x^1,x') = a^{ij}(\mu^{-1}x^1,x'),
\quad
\bar{u}(x^1, x') = u(\mu^{-1}x^1, x'),
\end{equation}
\begin{equation}							 \label{eq081402}
\tilde{f}(x^1,x') = f(\mu^{-1}x^1, x'),
\quad
\tilde{g}(x^1,x') = (\mu g_1, g_2, \cdots, g_d)(\mu^{-1}x^1,x').
\end{equation}
Clearly $\bar{u}$ satisfies
$$
\left(\mu^{2} \bar{a}^{11} \bar{u}_{x^1}\right)_{x^1}
+ \sum_{j >1}\left(\mu \bar{a}^{1j} \bar{u}_{x^1}\right)_{x^j}
+ \sum_{i >1}\left(\mu \bar{a}^{i1} \bar{u}_{x^i}\right)_{x^1}
+ \sum_{i,j >1}\left(\bar{a}^{ij} \bar{u}_{x^i}\right)_{x^j}
- \lambda \bar{u}
$$
$$
= \Div \tilde{g} + \tilde{f}.
$$
If we set $\bar{\cL_0} w =(\bar{a}^{11} w_{x^1})_{x^1} + \Delta_{d-1}w$,
where $\Delta_{d-1} w = \sum_{i=2}^d w_{x^ix^i}$,
then
\begin{equation}							 \label{eq081403}
\bar{\cL_0} \bar{u} - \mu^{-2} \lambda \bar{u} = \Div \bar{g}
+ \bar{f},
\end{equation}
where
$$
\bar{f} = \mu^{-2} \tilde{f},
\quad
\bar{g}_1 = \mu^{-2} \tilde{g}_1 - \mu^{-1}\sum_{i > 1} \bar{a}^{i1} \bar{u}_{x^i},
$$
$$
\bar{g}_j =  \mu^{-2} \tilde{g}_j - \mu^{-1}\bar{a}^{1j} \bar{u}_{x^1} - \mu^{-2} \sum_{i > 1} \bar{a}^{ij} \bar{u}_{x^i} + \bar{u}_{x^j}
\quad j \ge 2.
$$

We now assume that the coefficient $a^{11}$ is a measurable function of $x^1 \in \bR$.
Under this condition on $a^{11}$ (no regularity assumptions on $a^{ij}$ if $ij > 1$)
we prove an estimate for $\bar{a}^{11}\bar{u}_{x^1}$.

\begin{lemma}							\label{lemma081701}
Let $\lambda > 0$, $r > 0$, $\kappa > 8 K \delta^{-1}$, and $a^{11} = a^{11}(x^1)$.
Assume that $u \in W_{2, \text{loc}}^1$ and
$$
\cL_0 u -\lambda u = \Div g + f,
$$
where $f$, $g \in L_{2,\text{loc}}$.
Then there exists a constant $N= N(d, \delta, K)$ such that
$$
\left( | \bar{a}^{11}\bar{u}_{x^1} - \left( \bar{a}^{11}\bar{u}_{x^1} \right)_{B_r} |^2 \right)_{B_r}^{1/2}
\le N \big(\kappa^{-1} + \kappa^{d/2} \mu^{-1} \big)\left(|\bar{u}_{x^1}|^2\right)^{1/2}_{B_{\kappa r}}
$$
$$
+ N \kappa^{d/2} \left( |\bar{u}_{x'}|^2 + \lambda|\bar{u}|^2 + |\tilde{g}|^2
+ \lambda^{-1}|\tilde{f}|^2 \right)^{1/2}_{B_{\kappa r}}
$$
for all $\mu \ge 1$,
where $\bar{a}^{ij}$, $\bar{u}$, $\tilde{f}$, and $\tilde{g}$  are those defined in \eqref{eq081401} and \eqref{eq081402}.

In particular, if $\lambda = f = 0$, i.e., $\cL_0 u = \Div g$, we have
$$
\left( | \bar{a}^{11}\bar{u}_{x^1} - \left( \bar{a}^{11}\bar{u}_{x^1} \right)_{B_r} |^2 \right)_{B_r}^{1/2}
\le N \big(\kappa^{-1} + \kappa^{d/2} \mu^{-1} \big)\left(|\bar{u}_{x^1}|^2\right)^{1/2}_{B_{\kappa r}}
$$
$$
+ N \kappa^{d/2} \left( |\bar{u}_{x'}|^2 + |\tilde{g}|^2\right)^{1/2}_{B_{\kappa r}}
$$
for all $\mu \ge 1$.
\end{lemma}

\begin{proof}
The second inequality in the lemma follows easily from the first.
Indeed, if we write $\cL_0 u - \lambda u = \Div g - \lambda u$, by the first inequality
$$
\left( | \bar{a}^{11}\bar{u}_{x^1} - \left( \bar{a}^{11}\bar{u}_{x^1} \right)_{B_r} |^2 \right)_{B_r}^{1/2}
\le N \big(\kappa^{-1} + \kappa^{d/2} \mu^{-1} \big)\left(|\bar{u}_{x^1}|^2\right)^{1/2}_{B_{\kappa r}}
$$
$$
+ N \kappa^{d/2} \left( |\bar{u}_{x'}|^2 + \lambda|\bar{u}|^2 + |\tilde{g}|^2 \right)^{1/2}_{B_{\kappa r}}.
$$
Then letting $\lambda \searrow 0$ gives the result.

To prove the first inequality in the lemma,
recall that $\bar{u}$ satisfies (see \eqref{eq081403})
$$
\bar{\cL_0} \bar{u} - \lambda \bar{u} = \Div \bar{g}
+ \bar{f}_{\lambda},
$$
where $\lambda > 0$ and $\bar{f}_{\lambda} =  \bar{f} + (\mu^{-2} - 1) \lambda \bar{u}$.
Using Theorem \ref{theorem08061901}
we find $w \in W_2^1$ satisfying
$$
\bar{\cL_0} w - \lambda w = \Div \left(I_{B_{\kappa r}} \bar{g} \right)
+ I_{B_{\kappa r}} \bar{f}_{\lambda},
$$
where $I_{\Omega}$ is the indicator function of a set $\Omega$.
Then $v := \bar{u}- w$ satisfies
$$
\bar{\cL_0} v - \lambda v = \Div \left((1- I_{B_{\kappa r}}) \bar{g} \right)
+ (1-I_{B_{\kappa r}}) \bar{f}_{\lambda}.
$$
In particular, $\bar{\cL_0} v - \lambda v= 0$ in $B_{\kappa r}$.

Now we use the following change of variables.
Set
$$
y^1 = \phi(x^1):=\int_0^{x^1}\frac 1 {\bar{a}^{11}(r)}\,dr,
\quad
y^j = x^j,
\quad
j \ge 2.
$$
Since $\delta \le \bar{a}^{11} \le K$, we readily see that the inverse $\phi^{-1}$ exists,
$\phi$ is a bi-Lipschitz function,  and
\begin{equation}							 \label{eq081501}
K^{-1} \le \phi(t)/t \le \delta^{-1},
\quad
\delta \le \phi^{-1}(t)/t \le K
\end{equation}
for $t \neq 0$.
We define
$$
\bar{v}(y^1, y') = v(\phi^{-1}(y^1), y').
$$
We also define $r_1 = \sqrt{2} \delta^{-1} r$ and $\kappa_1 = \kappa/(2K\delta^{-1})$.
Using the fact that $\bar{\cL_0} v - \lambda v= 0$ in $B_{\kappa r}$,
$\kappa_1 r_1 = \kappa r/(\sqrt{2}K)$,
and \eqref{eq081501}, we find that, in $B_{\kappa_1 r_1}$,
$$
\bar{v}_{y^1y^1}
+ \hat{a}^{11}(y^1) \Delta_{d-1} \bar{v}
- \lambda \hat{a}^{11}(y^1) \bar{v}  = 0,
$$
where $\hat{a}^{11}(y^1) = \bar{a}^{11}(\phi^{-1}(y^1))$.
Equivalently, in $B_{\kappa_1 r_1}$,
$$
\Delta \bar{v} - \lambda \bar{v} = \left(1 - \hat{a}^{11}(y^1) \right) \Div \left(0, \bar{v}_{y^2}, \cdots, \bar{v}_{y^d} \right)
- \lambda \left(1 - \hat{a}^{11}(y^1) \right) \bar{v}.
$$
Then by using the change of variables as well as Theorem \ref{th080601} (note that $\kappa_1 \ge 4$)
we obtain
\begin{multline}							 \label{eq081503}
\dashint_{B_r} | \bar{a}^{11}v_{x^1} - \left( \bar{v}_{y^1} \right)_{B_{r_1}} |^2 \, dx
\le
N \dashint_{B_{r_1}} | \bar{v}_{y^1} - \left( \bar{v}_{y^1} \right)_{B_{r_1}} |^2 \, dy
\\
\le N \kappa_1^{-2} \left( |\bar{v}_y|^2 \right)_{B_{\kappa_1 r_1}}
+ N \kappa_1^{d} \left( |\bar{v}_{y'}|^2 + \lambda|\bar{v}|^2\right)_{B_{\kappa_1 r_1}}
\\
\le N \kappa^{-2} \left( |v_x|^2\right)_{B_{\kappa r}}
+ N \kappa^{d} \left( |v_{x'}|^2 + \lambda|v|^2\right)_{B_{\kappa r}},
\end{multline}
where $N = N(d,\delta,K)$.

We also need estimates for $w$.
By Theorem \ref{theorem08061901}
$$
\| w_{x} \|_{L_2}
+ \sqrt{\lambda} \| w \|_{L_2} \le N \left( \| I_{B_{\kappa r}} \bar{g} \|_{L_2}
+ \lambda^{-1/2}\| I_{B_{\kappa r}} \bar{f}_{\lambda} \|_{L_2} \right).
$$
From this and the definition of $\bar{f}_{\lambda}$ it follows that (also note that $\mu \ge 1$)
\begin{equation}							 \label{eq082005}
\left( |w_{x}|^2 \right)_{B_r}
\le N \kappa^d \left( |\bar{g}|^2 + \lambda^{-1}|\bar{f}|^2 + \lambda |\bar u|^2 \right)_{B_{\kappa r}},
\end{equation}
\begin{equation}							 \label{eq082004}
\left( |w_{x}|^2 + \lambda |w|^2 \right)_{B_{\kappa r}}
\le N \left( |\bar{g}|^2 + \lambda^{-1}|\bar{f}|^2 + \lambda |\bar u|^2 \right)_{B_{\kappa r}}.
\end{equation}
This together with $\bar{u} = w + v$ yields
\begin{equation}							 \label{eq082003}
\left( |v_{x'}|^2 + \lambda |v|^2 \right)_{B_{\kappa r}}
\le N \left( |\bar{u}_{x'}|^2 + \lambda |\bar{u}|^2 + |\bar{g}|^2 + \lambda^{-1}|\bar{f}|^2\right)_{B_{\kappa r}}.
\end{equation}

To combine all the inequalities shown above, we start with
$$
I := \left( | \bar{a}^{11}\bar{u}_{x^1} - \left( \bar{a}^{11}\bar{u}_{x^1} \right)_{B_r} |^2 \right)_{B_r}^{1/2}
\le \left( | \bar{a}^{11}\bar{u}_{x^1} - C |^2 \right)_{B_r}^{1/2},
$$
which holds true for any constant $C$.
Upon replacing $C$ with $\left( \bar{v}_{x^1} \right)_{B_{r_1}}$
and using $\bar{u} = w +v$ again,
we arrive at
$$
I \le \left( | \bar{a}^{11}\bar{u}_{x^1} - \left( \bar{v}_{x^1} \right)_{B_{r_1}} |^2 \right)_{B_r}^{1/2}
\le N \left( | \bar{a}^{11}v_{x^1} - \left( \bar{v}_{x^1} \right)_{B_{r_1}} |^2 \right)_{B_r}^{1/2}
+ N \left( |w_{x}|^2 \right)_{B_r}^{1/2}
$$
$$
=: I_1 + I_2.
$$
From \eqref{eq081503}, \eqref{eq082003}, and \eqref{eq082004}
$$
I_1 \le N \kappa^{-1} \left( |\bar{u}_x|^2 \right)_{B_{\kappa r}}^{1/2}
+ N \kappa^{d/2}\left( |\bar{u}_{x'}|^2 + \lambda |\bar{u}|^2 + |\bar{g}|^2 + \lambda^{-1}|\bar{f}|^2\right)^{1/2}_{B_{\kappa r}}.
$$
Here we also used $\bar{u} = w + v$ and $\kappa \ge 1$.
From \eqref{eq082005},
$$
I_2 \le N \kappa^{d/2}\left(\lambda |\bar{u}|^2 + |\bar{g}|^2 + \lambda^{-1}|\bar{f}|^2\right)^{1/2}_{B_{\kappa r}}.
$$
Finally, notice that
$$
\left(|\bar{g}|^2\right)_{B_{\kappa r}}^{1/2}
\le N \mu^{-2} \left(|\tilde{g}|^2\right)^{1/2}_{B_{\kappa r}}
+ N \mu^{-1}\left(|\bar{u}_{x^1}|^2\right)^{1/2}_{B_{\kappa r}}
+ N \left(|\bar{u}_{x'}|^2\right)^{1/2}_{B_{\kappa r}},
$$
$$
(|\bar{f}|^2)_{B_{\kappa r}}^{1/2}
= \mu^{-2} (|\tilde{f}|^2 )_{B_{\kappa r}}^{1/2}.
$$
Therefore,
$$
I \le N (\kappa^{-1} + \kappa^{d/2}\mu^{-1}) \left( |\bar{u}_{x^1}|^2 \right)_{B_{\kappa r}}^{1/2}
+ N \kappa^{d/2}\left( |\bar{u}_{x'}|^2 + \lambda |\bar{u}|^2 + |\tilde{g}|^2
+ \lambda^{-1}|\tilde{f}|^2 \right)^{1/2}_{B_{\kappa r}}
$$
for $\mu \ge 1$.
The lemma is proved.
\end{proof}

We recall the maximal function theorem and the Fefferman-Stein theorem.
Let the maximal and sharp functions of $g$ defined on $\bR^d$ be given by
$$
M g (x) = \sup_{r>0} \dashint_{B_r(x)} |g(y)| \, dy,
$$
$$
g^{\#}(x) = \sup_{r>0} \dashint_{B_r(x)} |g(y) -
(g)_{B_r(x)}| \, dy.
$$
Then
$$
\| g \|_{L_p} \le N \| g^{\#} \|_{L_p},
\quad
\| M g \|_{L_p} \le N \| g\|_{L_p},
$$
if $g \in L_p$, where $1 < p < \infty$ and $N = N(d,p)$.
As is well known, the first inequality above is due to the Fefferman-Stein theorem on sharp functions
and the second one is the Hardy-Littlewood maximal function theorem (this inequality also holds trivially when $p = \infty$).

Theorem \ref{th081201} below is from \cite{Krylov08}
and can be considered as a generalized version of the Fefferman-Stein Theorem.
To state this theorem,
let
$$
\bC_n = \{ C_n(i_1, \cdots, i_d), i_0, \cdots, i_d \in \bZ \},
\quad n \in \bZ
$$
be the collection of
partitions given by the dyadic cubes in $\bR^d$
$$
C_n(i_1, \cdots, i_d) = [ i_1 2^{-n}, (i_1+1)2^{-n} ) \times \cdots \times [ i_d 2^{-n}, (i_d+1)2^{-n} ).
$$

\begin{theorem}							\label{th081201}
Let $p \in (1, \infty)$, and $U,V,F\in L_{1,\text{loc}}$.
Assume that we have $|U| \le V$
and, for any $n \in \bZ$ and $C \in \bC_n$,
there exists a measurable function $U^C$ on $C$
such that $|U| \le U^C \le V$ on $C$ and
\begin{equation}							 \label{eq082006}
\min\left\{ \int_C |U - \left(U\right)_C| \, dx ,
\int_C |U^C - \left(U^C\right)_C| \, dx \right\}
\le \int_C F(x) \, dx.
\end{equation}
Then,
$$
\| U \|_{L_p}^p
\le N(d,p) \|F\|_{L_p}\| V \|_{L_p}^{p-1},
$$
provided that $F,V\in L_p$.
\end{theorem}

If $a^{11}$ is measurable in $x^1 \in \bR$ and has a locally small BMO semi-norm in $x' \in \bR^{d-1}$,
we show in the following lemma that $\bar{u}$, where $u$ is a solution to $\cL_0 u = \Div g$,
satisfies an inequality as in \eqref{eq082006}.

\begin{lemma}							\label{lem082001}
Let $\gamma>0$,
$\mu \ge 1$,
and $\tau$, $\sigma \in (1,\infty)$ such that $1/\tau + 1/\sigma = 1$.
Assume that $a^{11}$ satisfy Assumption \ref{assumption20080424a} ($\gamma$)
and $g \in L_{2,\text{loc}}$.
Also assume that $u \in W_{2, \text{loc}}^1$ vanishes outside $B_{\mu^{-1}R}$,
where $R\in (0,R_0]$,
and satisfies $\cL_0 u = \Div g$.
Then,
for each $C \in \bC_n$, $\mu \ge 1$,
and $\kappa > 8 K \delta^{-1}$,
there exists a measurable function $\bar{a}(x^1) = \bar{a}_{\mu,\kappa, C}(x^1)$ such that $\delta \le \bar{a}(x^1) \le K$
and
$$
\dashint_{C} | \bar{a} \bar{u}_{x^1} - \left(\bar{a} \bar{u}_{x^1}\right)_C | \, dx
\le N F(x)
$$
for all $x \in C$,
where $N = N(d, \delta, K)$ and
$$
F(x)= F_{\mu, \kappa}(x) = (\kappa^{-1} + \kappa^{d/2} \mu^{-1})\left( M |\bar{u}_{x^1}|^2\right)^{1/2}
+ \kappa^{d/2} \left( M |\bar{u}_{x'}|^2 \right)^{1/2}
$$
$$
+ \kappa^{d/2} \mu^{1/(2\sigma)} \gamma^{1/(2\sigma)}
\left( M |\bar{u}_{x^1}|^{2\tau} \right)^{1/(2\tau)}
+ \kappa^{d/2} \left( M |\tilde{g}|^2\right)^{1/2}.
$$
Recall that $\bar{u}$ and $\tilde{g}$ are those in \eqref{eq081401} and \eqref{eq081402}.
\end{lemma}

\begin{proof}
Let $B_r(x_0)$ be the smallest ball containing $C$.
We split into two cases depending on whether $\kappa r < R$ or $\kappa r \ge R$.

If $\kappa r < R$.
Set
$$
a(x^1) = \dashint_{B'_{\kappa r}(x_0')} a^{11}(x^1,y') \, dy',\quad \bar{a}(x^1) = a(\mu^{-1}x^1).$$
Since
$$
(a u_{x^1})_{x^1} + \sum_{ij>1} (a^{ij} u_{x^i})_{x^j}
= \Div g + \left((a - a^{11})u_{x^1}\right)_{x^1},
$$
by Lemma \ref{lemma081701} with an appropriate translation
$$
I := \left( | \bar{a}\bar{u}_{x^1} - \left( \bar{a}\bar{u}_{x^1} \right)_{B_r(x_0)} |^2 \right)_{B_r(x_0)}^{1/2}
\le N (\kappa^{-1} + \kappa^{d/2} \mu^{-1} )\left(|\bar{u}_{x^1}|^2\right)^{1/2}_{B_{\kappa r}(x_0)}
$$
$$
+ N \kappa^{d/2} \left( |\bar{u}_{x'}|^2 + |\tilde{g}|^2 \right)^{1/2}_{B_{\kappa r}(x_0)}
+ N \kappa^{d/2} \left(|(\bar{a} - \bar{a}^{11}) \bar{u}_{x^1}|^2\right)^{1/2}_{B_{\kappa r}(x_0)}.
$$
Note that
$$
\left(|(\bar{a} - \bar{a}^{11}) \bar{u}_{x^1}|^2\right)^{1/2}_{B_{\kappa r}(x_0)}
\le \left(| \bar{a} - \bar{a}^{11} |^{2\sigma}\right)^{1/(2\sigma)}_{B_{\kappa r}(x_0)}
\left(|\bar{u}_{x^1} |^{2\tau} \right)^{1/(2\tau)}_{B_{\kappa r}(x_0)},
$$
where
$$
\left(| \bar{a} - \bar{a}^{11} |^{2\sigma}\right)_{B_{\kappa r}(x_0)}
\le N \dashint_{B_{\kappa r}(x_0)}
| \bar{a} - \bar{a}^{11} | \, dx
$$
$$
\le N \dashint_{x_0^1 - \kappa r}^{\,\,\,\,x_0^1 + \kappa r}
\dashint_{B'_{\kappa r}(x'_0)} \big|a^{11}(\mu^{-1}x^1, x') - \dashint_{B'_{\kappa r}(x'_0)} a^{11}(\mu^{-1}x^1,y') \, dy' \big| \, dx' \, dx^1
$$
$$
\le N \mu \, \text{osc}_{x'}\left(a^{11},\Gamma_{\kappa r}(\mu^{-1}x^1_0, x'_0)\right)
\le N \mu a^{\#}_{\kappa r} \le N \mu a^{\#}_{R}\le N\mu\gamma.
$$
Also note that if $x \in C$, then $B_{2 \kappa r}(x) \supset B_{\kappa r}(x_0)$
and, for example,
$$
\left( |\bar{u}_{x^1}|^2\right)_{B_{\kappa r}(x_0)}
\le 2^{d} \left( |\bar{u}_{x^1}|^2\right)_{B_{2\kappa r}(x)}
\le 2^{d} \left(M |\bar{u}_{x^1}|^2 (x) \right)
$$
for all $x \in C$.
From this observation as well as the above inequalities for $I$,
we obtain $I \le N F_{\mu,\kappa}(x)$ for all $x \in C$.

If $\kappa r \ge R$.
Set
$$
a(x^1) = \dashint_{B'_R} a^{11}(x^1,y') \, dy',
\quad\bar{a}(x^1) = a(\mu^{-1}x^1).$$
Since $u$ vanishes outside $B_{\mu^{-1}R}$,
$\bar{u}$ has a compact support in $B_R$.
Thus
$$
\left(|(\bar{a} - \bar{a}^{11}) \bar{u}_{x^1}|^2\right)_{B_{\kappa r}(x_0)}
= \frac{1}{|B_{\kappa r}|}\int_{B_{\kappa r}(x_0) \cap B_{R}}
|(\bar{a} - \bar{a}^{11}) \bar{u}_{x^1}|^2 \, dx
$$
$$
\le \left(\frac{1}{|B_{\kappa r}|} \int_{B_R} |\bar{a} - \bar{a}^{11}|^{2\sigma} \, dx\right)^{1/\sigma}
\left(|\bar{u}_{x^1}|^2\right)^{1/\tau}_{B_{\kappa r}(x_0)},
$$
where
$$
\frac{1}{|B_{\kappa r}|} \int_{B_R} |\bar{a} - \bar{a}^{11}|^{2\sigma} \, dx
$$
$$
\le N \frac{1}{|B_{\kappa r}|} \int_{-R}^{R}
\int_{B'_R} \big|a^{11}(\mu^{-1}x^1, x') - \dashint_{B'_R} a^{11}(\mu^{-1}x^1,y') \, dy' \big| \, dx' \, dx^1
$$
$$
\le N \mu (\kappa r)^{-d}R^d \, \text{osc}_{x'}\left(a^{11},\Gamma_{R}(0)\right)
\le N \mu a^{\#}_{R} \le N \mu \gamma.
$$
If we proceed as in the first case, we come to $I \le N F_{\mu,\kappa}(x)$ for all $x \in C$.

Finally,
observe that
$$
\dashint_{C} | \bar{a} \bar{u}_{x^1} - \left(\bar{a} \bar{u}_{x^1}\right)_C | \, dx
\le 2 \dashint_{C} | \bar{a} \bar{u}_{x^1} - \left( \bar{a} \bar{u}_{x^1} \right)_{B_r(x_0)} | \, dx
\le N I,
$$
where $N$ is independent of $r$. The lemma is proved.
\end{proof}

Now we are ready to prove that
the $L_p$-norm of $u_{x^1}$ is controlled
by that of $g$ and $u_{x'}$
if $u$ is a solution to $\cL_0 u= \Div g$
with $a^{11}$ measurable in $x^1 \in \bR$ and small BMO in $x' \in \bR^{d-1}$.

\begin{theorem}
                                \label{theorem5.3}
Let $p\in (2,\infty)$ and $g\in L_p$.
There exist constants $\gamma$,  $\mu$, and $N$, depending on $d,p,\delta$ and $K$,
such that, if $a^{11}$ satisfies Assumption \ref{assumption20080424a} ($\gamma$), then for $u\in C_0^\infty$
satisfying $\cL_0 u=\Div g$
and vanishing outside $B_{\mu^{-1} R}$,
where   $R \le R_0$, we have
$$
\|u_{x}\|_{L_p}\leq N(\|u_{x'}\|_{L_p}+\|g\|_{L_p}).
$$
\end{theorem}

\begin{proof}
It is enough to prove
$$
\|u_{x^1}\|_{L_p}\leq N(\|u_{x'}\|_{L_p}+\|g\|_{L_p}).
$$
Fix $\tau$ in Lemma \ref{lem082001} such that $p > 2 \tau > 2$.
Also take $\kappa > 8K\delta^{-1}$ and $\mu \ge 1$ to be specified below.
To use Theorem \ref{th081201},
we set $U = \delta \bar{u}_{x^1}$
and $V = K |\bar{u}_{x^1}|$, where $\bar{u}$ is from Lemma \ref{lem082001}.
For each $C \in \bC_n$, we set $U^C = |\bar{a} \bar{u}_{x^1}|$,
where $\bar{a}  = \bar{a}_{\mu,\kappa,C}$ is also from Lemma \ref{lem082001}.
Since $\delta \le \bar{a} \le K$, we have
$$
|U| \le U^C \le V.
$$
Note that
$$
\int_C | U^C - \left(U^C\right)_C | \, dx
\le 2 \int_C | \bar{a} \bar{u}_{x^1} - \left(\bar{a} \bar{u}_{x^1}\right)_C | \, dx
\le N \int_C F_{\mu,\kappa}(x) \, dx,
$$
where the second inequality is due to Lemma \ref{lem082001}.
Then by Theorem \ref{th081201},
$$
\| \bar{u}_{x^1} \|^p_{L_p}
\le N \| F_{\mu,\kappa} \|_{L_p} \| \bar{u}_{x^1} \|^{p-1}_{L_p}.
$$
From this and using the maximal function theorem
(it is essential to have $p> 2\tau $)
we get
$$
\| \bar{u}_{x^1} \|_{L_p}
\le N \| F_{\mu,\kappa} \|_{L_p} \le N \kappa^{d/2} \|\bar{u}_{x'}\|_{L_p}
+ N \kappa^{d/2} \|\tilde{g}\|_{L_p}
$$
$$
+ N \left(\kappa^{-1} + \kappa^{d/2} \mu^{-1} + \kappa^{d/2} \mu^{1/(2\sigma)} \gamma^{1/(2\sigma)}\right) \|\bar{u}_{x^1}\|_{L_p},
$$
where $N = N(d,p,\delta,K)$.
Choose a sufficiently big $\kappa$, then  $\mu$, and finally a small $\gamma$ so that
$$
N \left(\kappa^{-1} + \kappa^{d/2} \mu^{-1} + \kappa^{d/2} \mu^{1/(2\sigma)} \gamma^{1/(2\sigma)}\right)
\le 1/2.
$$
Then
$$
\| \bar{u}_{x^1} \|_{L_p}
\le N \left( \| \bar{u}_{x'} \|_{L_p} + \| \tilde{g} \|_{L_p} \right).
$$
To finish the proof, we just return to $u$ and $g$ by using \eqref{eq081401} and \eqref{eq081402}.
\end{proof}

\mysection{Equations in divergence form with simple leading coefficients}
                                        \label{ellSec1}

In this section, we set
$$
\bar{\cL} u = (a^{ij}u_{x^i})_{x^j},
$$
where the coefficients are measurable functions of $x^1 \in \bR$ only, i.e., $a^{ij}=a^{ij}(x^1)$.
We denote, as usual,
$$
[ u ]_{\alpha, \Omega}
= \sup_{x, y\in \Omega} \frac{|u(x) - u(y)|}{|x-y|^{\alpha}}.
$$

\begin{lemma}				\label{lemma080701}
Let $p \in [1, \infty)$, $\lambda\ge 0$.
Assume $u \in C_{\text{loc}}^{\infty}$ and $\bar{\cL} u-\lambda u=0$ in $B_{2}$.
Then we have
$$
\left[u_{x'} \right]_{\alpha,B_1} \le N \left( \|u_{x}\|_{L_p(B_2)} + \lambda^{1/2} \|u\|_{L_p(B_2)} \right),
$$
where $(N, \alpha) = (N, \alpha)(d, p, \delta,K)$.
\end{lemma}
\begin{proof}
First assume that $\lambda = 0$.
By the De Giorgi-Moser-Nash H\"{o}lder estimate,
there exist $N$ and $\alpha \in (0,1)$, depending only on $d$, $p$, $\delta$, and $K$, such that
$$
\left[ u \right]_{\alpha, B_1} \le N \| u \|_{L_p(B_{2})}.
$$
Note that $u_{x'}$ also satisfies $\bar{\cL}u_{x'} = 0$ in $B_{2}$.
Thus
$$
\left[ u_{x'} \right]_{\alpha, B_1} \le N \| u_{x'} \|_{L_p(B_{2})}.
$$

If $\lambda > 0$, we use an idea by S. Agmon.
Let $z = (x,y)$ be a point in $\bR^{d+1}$, where $x \in \bR^{d}$, $y \in \bR$,
and $\hat{u}(z)$ and $\hat{B}_r$ be given by
$$
\hat{u}(z) = \hat{u}(x, y) = u(x) \cos(\sqrt{\lambda} y),
\quad
\hat{B}_r = \{ |z| < r: z \in \bR^{d+1} \}.
$$
Since $\hat{u}$ satisfies, in $\hat{B}_{2}$,
$$
\bar{\cL}\hat{u} +(\hat{u}_{y})_{y} = 0,
$$
by the above result applied to $\hat{u}$ we have
\begin{equation}								 \label{eq0804}
\left[ \hat{u}_{x'} \right]_{\alpha, \hat{B}_1}
\le N \|\hat{u}_z\|_{L_p(\hat{B}_2)}
\end{equation}
where $N = N(d,p, \delta,K)$.
Observe that
$$
\left[ u_{x'} \right]_{\alpha, B_1}
\le \left[ \hat{u}_{x'} \right]_{\alpha, \hat{B}_1}
$$
and $D_z \hat{u}$ is the collection consisting of
$$
\cos( \sqrt{\lambda} y) u_{x},
\quad
-\sqrt{\lambda} \sin (\sqrt{\lambda} y ) u.
$$
Thus the right-hand side of \eqref{eq0804} is less than the right-hand side of the inequality in the lemma.
The lemma is proved.
\end{proof}

\begin{corollary}			\label{cor080702}
Let $p \in [1, \infty)$, $\kappa\ge 2$, $r>0$, and $\lambda\ge 0$.
Assume $u \in C_{\text{loc}}^{\infty}$ and $\bar{\cL} u-\lambda u=0$ in $B_{\kappa r}$.
Then we have
$$
\left(|u_{x'}-\left(u_{x'}\right)_{B_r}|^p\right)_{B_r}\le N \kappa ^{-p\alpha}\left(|u_x|^p+\lambda^{p/2} |u|^p\right)_{B_{\kappa r}},
$$
where $(N, \alpha) = (N,\alpha)(d, p, \delta,K)$.
\end{corollary}

\begin{proof}
Thanks to the scaling argument it is enough to prove the estimate when $r = 1$.
Set
$$
\hat{a}^{ij}(x) = a^{ij}(\kappa x/2),
\quad
v(x)=u(\kappa x/2).
$$
Then $v$ satisfies, in $B_{2}$,
$$
\left( \hat{a}^{ij} v_{x^i} \right)_{x^j} - (\kappa/2)^2 \lambda v = 0.
$$
By Lemma \ref{lemma080701},
$$
\left[v_{x'} \right]_{\alpha,B_1} \le N \left( \|v_{x}\|_{L_p(B_2)} + \kappa \lambda^{1/2} \|v\|_{L_p(B_2)} \right)
$$
$$
\le N \kappa \left( |u_x|^p \right)^{1/p}_{B_{\kappa}}
+ N \kappa \lambda^{1/2}\left( |u|^p \right)^{1/p}_{B_{\kappa}}.
$$
Note that
$$
\left[v_{x'} \right]_{\alpha,B_1} = (\kappa/2)^{1+\alpha} \left[u_{x'} \right]_{\alpha,B_{\kappa/2}}.
$$
Using this and the above inequality, we see that
$$
\left(|u_{x'}-\left(u_{x'}\right)_{B_1}|^p\right)_{B_1}
\le N \left[u_{x'} \right]^p_{\alpha,B_{\kappa/2}}
\le N \kappa^{-p\alpha}
\left( |u_x|^p + \lambda^{p/2} |u|^p \right)_{B_{\kappa}}.
$$
\end{proof}

We prove a version of Theorem \ref{th080601} when $p = 2$ and the Laplace operator is replaced by $\bar{\cL}$.
However, due to the fact that $a^{ij}$ are measurable with respect to $x^1 \in \bR$,
we only have the estimate of the $L_2$-oscillations of $u_{x'}$.
In the proof we use Corollary \ref{cor080702} for $p = 2$.

\begin{theorem}
                                        \label{thm2.05}
Let $\lambda > 0$, $\kappa\ge 4$, $r>0$, $u\in W^1_{2,\text{loc}}$ and $f$, $g\in L_{2,\text{loc}}$.
Assume that
$$
\bar{\cL}u - \lambda u =\Div g + f
$$
in $B_{\kappa r}$.
Then there exist positive constants $N$ and $\alpha$, depending only on $d$, $\delta$, and $K$, such that
\begin{equation}
                                                    \label{eq2.06}
\left(|u_{x'}-(u_{x'})_{B_r}|^2\right)_{B_r}\leq N\kappa^{-2\alpha}
\left(|u_{x}|^2 + \lambda |u|^2\right)_{B_{\kappa r}}+N\kappa^{d}\left(|g|^2 + \lambda^{-1}|f|^2\right)_{B_{\kappa r}}.
\end{equation}
In particular, if $\lambda = f = 0$, i.e., $\bar{\cL} u = \Div g$, we have
$$
\left(|u_{x'}-(u_{x'})_{B_r}|^2\right)_{B_r}\leq N\kappa^{-2\alpha}
\left(|u_{x}|^2\right)_{B_{\kappa r}}+N\kappa^{d}\left(|g|^2\right)_{B_{\kappa r}}.
$$
\end{theorem}

\begin{proof}
As in the proof of Lemma \ref{lemma081701},
it suffices to prove \eqref{eq2.06}.
We proceed adopting the same strategy as in the proof of Theorem \ref{th080601}.
As noted there, we can assume that all the coefficients as well as $u$, $f$, and $g$ are infinitely differentiable.

Take a $\zeta \in C_0^{\infty}$ such that
$$
\zeta = 1 \quad \text{on} \quad B_{\kappa r/2},
\quad
\zeta = 0 \quad \text{on}
\quad
\bR^d \setminus B_{\kappa r}.
$$
By Theorem \ref{theorem08061901}, for $\lambda > 0$, there exists a unique solution $w \in W_2^1$
to the equation
$$
\left( \bar{\cL} - \lambda \right) w = \Div ( \zeta g ) + \zeta f.
$$
Since all functions and coefficients involved are infinitely differentiable,
by the classical theory on elliptic equations in divergence form,
$w$ is infinitely differentiable.
The function $v := u - w$ is also infinitely differentiable and satisfies
$$
\left( \bar{\cL} - \lambda \right) v = \Div ( (1-\zeta) g ) + (1- \zeta) f,
$$
as well as $\left( \bar{\cL} - \lambda \right) v = 0$ in $B_{\kappa r/2}$.
Thus by Corollary \ref{cor080702} (note that $\kappa /2 \ge 2$)
\begin{equation}							\label{eq1003}
\left(|v_{x'}-\left(v_{x'} \right)_{B_r}|^2\right)_{B_r} \le N \kappa^{-2\alpha}\left(|v_x|^2+\lambda |v|^2\right)_{B_{\kappa r}}.
\end{equation}
Regarding $w$, by Theorem \ref{theorem08061901} we have
$$
\| w_x \|_{L_2}
+ \sqrt{\lambda}  \| w \|_{L_2}
\le N \left(\| \zeta g \|_{L_2} + \lambda^{-1/2}\| \zeta f \|_{L_2}\right),
$$
In particular,
\begin{equation}							\label{eq1001}
\left(|w_x|^2\right)_{B_r}
\le N r^{-d} \left(\| \zeta g \|^2_{L_2} + \lambda^{-1}\| \zeta f \|^2_{L_2}\right)
\le N \kappa^{d} \left( |g|^2
+ \lambda^{-1} |f|^2 \right)_{B_{\kappa r}},
\end{equation}
\begin{equation}							\label{eq1004}
\left(|w_x|^2\right)_{B_{\kappa r}}
+ \lambda  \left(|w|^2\right)_{B_{\kappa r}}
\le N \left( |g|^2 + \lambda^{-1} |f|^2 \right)_{B_{\kappa r}},
\end{equation}

Now we prove \eqref{eq2.06}.
From \eqref{eq1003}, \eqref{eq1001}, and the fact that $u = w + v$,
we obtain
$$
\left(| u_{x'} - \left( u_{x'} \right)_{B_r} |^2\right)_{B_r}
\le N\left(|v_{x'}-(v_{x'})_{B_r}|^2\right)_{B_r}
+ N\left(|w_{x'}|^2\right)_{B_r}
$$
$$
\le N \kappa^{-2\alpha}\left(|v_x|^2+\lambda |v|^2\right)_{B_{\kappa r}}
+ N \kappa^{d} \left( |g|^2 + \lambda^{-1}|f|^2\right)_{B_{\kappa r}}.
$$
From \eqref{eq1004}, we also get
$$
\left(|v_x|^2+\lambda |v|^2\right)_{B_{\kappa r}}
\le N \left(|u_x|^2+\lambda |u|^2\right)_{B_{\kappa r}}
+ N \left(|{w}_x|^2+\lambda |w|^2\right)_{B_{\kappa r}}
$$
$$
\le N \left(|u_x|^2+\lambda |u|^2\right)_{B_{\kappa r}}
+ N \left( |g|^2 + \lambda^{-1}|f|^2 \right)_{B_{\kappa r}}.
$$
Combining the above two sets of inequalities we come to the inequality \eqref{eq2.06}.
\end{proof}

\mysection{Equations with partially small BMO coefficients}
                                        \label{ellSec2}
We prove in this section Theorem \ref{th081901}, the first of our main results, where we consider the operator $\cL$ with coefficients in their full generality as given by Assumption \ref{assumption20080424}.
That is, we consider
$$
\cL u=(a^{ij}u_{x^i} + a^{j} u)_{x^j} + b^{i}u_{x^i}+cu,
$$
where $a^{ij}$ are measurable in $x^1$ and  have locally small BMO semi-norms in $x' \in \bR^{d-1}$.
All the other coefficients $a^{i}$, $b^{i}$, and $c$ are only bounded and measurable.

\begin{theorem}
                            \label{theorem3.42}
Let $a^i=b^i=c=0$,  $\gamma > 0$, $\tau,\sigma \in (1,\infty)$, $1/\tau+1/\sigma=1$, and $R\in (0,R_0]$.
Assume $u\in C_0^\infty$ vanishing outside $B_R$ and $\cL u=\Div g$,
where $g\in L_2$.
Then under Assumption \ref{assumption20080424} ($\gamma$) there exists a positive constant $N$, depending only on $d$, $\delta$, $K$, and $\tau$, such that
\begin{multline}
                                                    \label{eq3.49}
\left(|u_{x'}-(u_{x'})_{B_r(x_0)}|^2\right)_{B_r(x_0)}\leq N\kappa^{-2\alpha}
\left(|u_{x}|^2\right)_{B_{\kappa r}(x_0)}
\\
+N\kappa^d\left( (|g|^2)_{B_{\kappa r}(x_0)}+\gamma^{1/\sigma}
(|u_x|^{2\tau})_{B_{\kappa r}(x_0)}^{1/\tau}\right),
\end{multline}
for any $r\in (0,\infty)$, $\kappa\ge 4$, and $x_0\in \bR^d$,
where $\alpha = \alpha(d, \delta, K) > 0$.
\end{theorem}

\begin{proof}
The proof is similar to that of Lemma \ref{lem082001}.
Fix $\kappa\ge 4$, $r\in (0,\infty)$, and $x_0 = (x_0^1,x'_0)\in \bR^d$.
Then introduce, for all $i,j = 1, \cdots, d$,
$$
\sba^{ij}(x^1)= \dashint_{B'_{\kappa r}(x_0')} a^{ij}(x^1, y')\, dy'
\quad
\text{if}
\quad
\kappa r<R,
$$
$$
\sba^{ij}(x^1)=\dashint_{B'_{R}} a^{ij}(x^1, y') \, d y'
\quad
\text{if}
\quad
\kappa r\ge R,
$$
and $\bar{\cL} u = \left( \sba^{ij} u_{x^i} \right)_{x^j}$.
We see that $\bar{\cL} u = \Div \hat{g}$,
where
$$
\hat{g_j} =  \left(\sba^{ij} - a^{ij}\right) u_{x^i} + g_j.
$$
Then by Theorem \ref{thm2.05} with an appropriate translation,
\begin{equation}							 \label{eq081101}
\dashint_{B_r(x_0)} | u_{x'} - \left( u_{x'} \right)_{B_r(x_0)} |^2 \, dx
\le N \kappa^{-2\alpha} \left( |u_{x'}|^2 \right)_{B_{\kappa r}(x_0)}
+ N \kappa^{d} \left( |\hat{g}|^2 \right)_{B_{\kappa r}(x_0)},
\end{equation}
where $(N,\alpha) = (N,\alpha)(d,\delta,K)$.
Observe that
\begin{equation}							 \label{eq081102}
\int_{B_{\kappa r}(x_0)} |\hat{g}|^2 \, dx
\le N \int_{B_{\kappa r}(x_0)} |g|^2 \, dx
+ N I,
\end{equation}
where $N = N(d)$ and
$$
I =
\int_{B_{\kappa r}(x_0)}
\big| (\sba^{ij} - a^{ij}) u_{x^i} \big|^2 \, dx
= \int_{B_{\kappa r}(x_0) \cap B_R}
\big| (\sba^{ij} - a^{ij}) u_{x^i} \big|^2 \, dx.
$$
By the H\"{o}lder's inequality, we have
\begin{equation}							 \label{eq081103}
I \le J_1^{1/\sigma} J_2^{1/\tau},
\end{equation}
where
$$
J_1 = \int_{B_{\kappa r}(x_0) \cap B_R} | \sba^{ij} - a^{ij} |^{2\sigma} \, dx,
\quad
J_2 = \int_{B_{\kappa r}(x_0)} |u_{x}|^{2\tau} \, dx.
$$
If $\kappa r < R$,
$$
J_1 \le N \int_{B_{\kappa r}(x_0)} | \sba^{ij} - a^{ij} | \, dx
$$
$$
\le N \int_{x_0^1-\kappa r}^{x_0^1+\kappa r}
\int_{B'_{\kappa r}(x'_0)} \big| a^{ij}(x^1, x')
- \dashint_{B'_{\kappa r}(x'_0)} a^{ij} (x^1, z') \, dz'
\big| \, dx' \, dx^1
$$
$$
\le N (\kappa r)^d a^{\#}_{\kappa r}
\le  N (\kappa r)^d a^{\#}_{R},
$$
where $N$ depends only on $d$ and $K$.
In case $\kappa r \ge R$,
$$
J_1 \le N \int_{B_R} | \sba^{11} - a^{11} | \, dx
$$
$$
\le N \int_{-R}^{R}
\int_{B'_R} \big| a^{11} (x^1, x') - \dashint_{B'_R} a^{11} (x^1, z') \, dz' \big| \, dx' \, dx^1
$$
$$
\le N R^d a^{\#}_{R}
\le  N (\kappa r)^d a^{\#}_{R},
$$
where $N = N(d,K)$.
From the above estimates for $J_1$ as well as the inequalities \eqref{eq081101}, \eqref{eq081102},
and  \eqref{eq081103}, we prove
$$
\dashint_{B_r(x_0)} | u_{x'} - \left( u_{x'} \right)_{B_r(x_0)} |^2 \, dx
\le N \kappa^{-2\alpha} \left( |u_{x}|^2 \right)_{B_{\kappa r}(x_0)}
+ N \kappa^{d} \left( |g|^2 \right)_{B_{\kappa r}(x_0)}
$$
$$
+ N \kappa^{d} 
\big(a^{\#}_R\big)^{1/\sigma}
\left( |u_{x}|^{2\tau} \right)^{1/\tau}_{B_{\kappa r}(x_0)},
$$
where $N= N(d,\delta,K,\sigma)$.
It only remains to notice that $a^{\#}_R \le \gamma$.
\end{proof}

\begin{lemma}
                                \label{lem3.52}
Let $p\in (2,\infty)$, $a^i = b^i = c=0$, $g \in L_p$ and $\mu$ be the constant in Theorem \ref{theorem5.3}.
Then there exist positive constants $\gamma$ and $N$ depending on $d$, $p$, $\delta$ and $K$
such that, under Assumption \ref{assumption20080424}($\gamma$), for $u\in C_0^\infty$ vanishing outside $B_R$, $R \le \mu^{-1}R_0$
and satisfying $\cL u=\Div g$, we have
$$
\|u_x\|_{L_p}\leq N\|g\|_{L_p}.
$$
\end{lemma}

\begin{proof}
Choose $\tau \in (1, \infty)$ such that $p > 2\tau$
and set
$$
\cA(x) = M(|g|^2)(x),
\quad
\cB(x) = M(|u_x|^2)(x),
\quad
\cC(x) = M(|u_x|^{2\tau}).
$$
Then
the inequality \eqref{eq3.49}
implies
$$
\left(| u_{x'} - (u_{x'})_{B_r(x_0)}|^2 \right)_{B_r(x_0)}
\le N \kappa^{d} \cA(x_0)
+ N \kappa^{-2\alpha}  \cB(x_0)
+ N \kappa^{d}\gamma^{1/\sigma} \cC(x_0)^{1/\tau}
$$
for all $x_0 \in \bR^d$, $\kappa \ge 4$, and $r > 0$.
Taking the supremum of the left-hand side of the above inequality with respect to $r > 0$
and using
$$
\left(| u_{x'} - (u_{x'})_{B_r(x_0)}| \right)_{B_r(x_0)}^2
\le \left(| u_{x'} - (u_{x'})_{B_r(x_0)}|^2 \right)_{B_r(x_0)},
$$
we obtain the following pointwise estimate:
$$
\big( u_{x'}^{\#}(x)\big)^2
\le N \kappa^{d} \cA(x) + N \kappa^{-2\alpha} \cB(x)
+  N \kappa^{d}\gamma^{1/\sigma} \cC(x)^{1/\tau}
$$
for all $x \in \bR^d$ and $\kappa \ge 4$.
Again apply the Fefferman-Stein theorem on sharp functions
and the Hardy-Littlewood maximal function theorem
on the above inequality to get
$$
\| u_{x'} \|_{L_p}
\le N \| u_{x'}^{\#} \|_{L_p}
\leq N \kappa^{d/2} \| M(|g|^2) \|_{L_{p/2}}^{1/2}
$$
$$
+ N \kappa^{-\alpha} \|M(|u_x|^2) \|_{L_{p/2}}^{1/2}
+ N \kappa^{d/2}\gamma^{1/(2\sigma)} \|M(|u_x|^{2\tau}) \|_{L_{p/(2\tau)}}^{1/(2\tau)}
$$
$$
\le N \kappa^{d/2} \| g \|_{L_p} + N \left(\kappa^{-\alpha} + \kappa^{d/2} \gamma^{1/(2\sigma)}\right)
\|u_x\|_{L_{p}},
$$
where the last inequality is possible due to $p > 2\tau>2$.
On the other hand, since $R \le \mu^{-1}R_0$, by Theorem \ref{theorem5.3}
 we have
$$
\|u_{x}\|_{L_p}\leq N(\|u_{x'}\|_{L_p}+\|g\|_{L_p})
$$
as long as $\gamma$ is less than the constant with the same notation in Theorem \ref{theorem5.3}.
Therefore,
$$
\| u_x \|_{L_p}
\le N \kappa^{d/2} \| g \|_{L_p}
+ N \left(\kappa^{-\alpha} + \kappa^{d/2} \gamma^{1/(2\sigma)}\right) \| u_x \|_{L_p},
$$
where $N = N(d,p,\delta,K)$.
Now we finish the proof by choosing a big enough $\kappa$ and then a possibly smaller $\gamma$
so that
$$
N \left(\kappa^{-\alpha} + \kappa^{d/2} \gamma^{1/(2\sigma)}\right) \le 1/2.
$$
\end{proof}

We  now conclude this section by proving Theorem \ref{th081901}.

\begin{proof}[Proof of Theorem \ref{th081901}]
To prove the first two assertions, by the method of continuity it is enough to prove the estimate.
Moreover, due to the duality argument we only need to consider the case $p \in (2, \infty)$.
Then the estimate in the theorem follows from Lemma \ref{lem3.52},
a partition of unity,
and the idea of Agmon shown, for example, in \cite{Krylov_2005}.
The last assertion is a consequence of the first two via a scaling argument.
\end{proof}

\mysection{Equations on a half space}
                                                \label{sec6}
This section is devoted to the proofs of Theorem \ref{thm6.6} and \ref{thm6.7}. We shall establish the solvability of divergence form elliptic equations on the half space  $\bR^d_+$ with either the Dirichlet boundary condition or the conormal derivative boundary condition by using the idea of odd/even extensions.

We will use the following well known results.

\begin{lemma}
                                \label{lem8.2}
Let $p,q\in (1,\infty)$.

\noindent (i) A function $u$ belongs to $W^{1}_{q,p}(\bR^d_+)$ if and only if its even extension $\tilde u$ with respect to $x_1$ belongs to $W^{1}_{q,p}(\bR^d)$. Moreover, we have,
\begin{equation}
                                    \label{eq2.27pm}
\|u\|_{L_{q,p}(\bR^d_+)}\le
\|\tilde u\|_{L_{q,p}(\bR^d)}\le 2
\|u\|_{L_{q,p}(\bR^d_+)},
\end{equation}
\begin{equation}
                                    \label{eq2.28pm}
\|u_x\|_{L_{q,p}(\bR^d_+)}\le
\|\tilde u_x\|_{L_{q,p}(\bR^d)}\le 2
\|u_x\|_{L_{q,p}(\bR^d_+)}.
\end{equation}

\noindent (ii) A function $u$ belongs to $W^{1}_{q,p}(\bR^d_+)$ and vanishes on $\partial \bR^d_+$ if and only if its odd extension $\tilde u$ with respect to $x_1$ belongs to  $W^{1}_{q,p}(\bR^d_+)$. Moreover, we have \eqref{eq2.27pm} and \eqref{eq2.28pm}.
\end{lemma}

Now we are ready to prove Theorem \ref{thm6.6} and \ref{thm6.7}.
\begin{proof}[Proof of Theorem \ref{thm6.6}]
We define
$$
\tilde a^{ij}(x)={\text{sgn}(x^1)}a^{ij}(|x^1|,x')\quad \text{for}\,\,i=1,j\ge 2\,\,\text{or}\,\,j=1,i\ge 2,
$$
$$
\tilde a^{ij}(x)=a^{ij}(|x^1|,x')\quad \text{otherwise},
$$
and
$$
\tilde a^1(x)={\text{sgn}(x^1)}a^1(|x^1|,x'),\quad \tilde a^j(x)=a^j(|x^1|,x'),\,\,j\ge 2,
$$
$$
\tilde {b}^1(x)={\text{sgn}(x^1)} b^1(|x^1|,x'),\quad \tilde {b}^j(x)=b^j(|x^1|,x'),\,\,j\ge 2,
$$
$$
\tilde c(x)=c(|x^1|,x'),\quad \tilde f(x)= {\text{sgn}(x^1)} f(|x^1|,x'),
$$
$$
\tilde g_1(x)=g_1(|x^1|,x'),\quad \tilde g_j(x)={\text{sgn}(x^1)} g_j(|x^1|,x'),\,\,j\ge 2.
$$
It is easily seen that if the original coefficients satisfy Assumption \ref{assumption20080424} ($\gamma$), then the new coefficients $\tilde a^{ij}$ satisfy Assumption \ref{assumption20080424} ($2\gamma$). Moreover, we have $\tilde f,\tilde g\in L_{ p}(\bR^d)$. Let $\tilde L$ be the divergence form elliptic operator with coefficients $\tilde a^{ij},\tilde a^i,\tilde {b}^i,\tilde c$. Due to Theorem \ref{th081901}, we can find $\gamma>0$ and $\lambda_0\ge 0$ such that there exists a unique solution $u\in W^1_{ p}$ solving
\begin{equation}
                                            \label{eq24.3.29}
\tilde Lu-\lambda u=\Div \tilde g+\tilde f\quad \text{in}\,\,\bR^d,
\end{equation}
provided that $\lambda> \lambda_0$. By the definition of the coefficients and the data, we have
$$
\tilde Lu(-x^1,x')-\lambda u(-x^1,x')=-\Div \tilde g(x)-\tilde f(x)\quad \text{in}\,\,\bR^d.
$$
Consequently, $-u(-x^1,x')$ is also a solution to \eqref{eq24.3.29}. By the uniqueness of the solution, we obtain $u(x)=-u(-x^1,x')$. This implies that, as a function on $\bR_+^d$, $u$ has zero trace on the boundary and clearly $u$ satisfies \eqref{eq2008061901} in $\bR_+^d$. This proves the existence of the solution.

To prove the uniqueness, let $v$ be another solution of \eqref{eq2008061901} so that, for any $\phi\in W^1_{p'}(\bR_+^d)$ with zero trace on $\partial\bR_+^d$, we have
\begin{equation}
                                                \label{eq3.2.36}
\int_{\bR_+^d}-a^{ij}v_{x^i}\phi_{x^j}-a^jv \phi_{x^j}+b^i v_{x^i}\phi+(c-\lambda)v\phi\,dx=\int_{\bR_+^d}-g_j\phi_{x^j}+f\phi\,dx.
\end{equation}
Denote $\tilde v$ to be the odd extension of $v$ with respect to $x^1$. Then by the definition of $\tilde a^{ij}$, $\tilde a^i$, $\tilde {b}^i$, $\tilde c$, $\tilde g$,  and $\tilde f$, for any $\varphi\in W^1_{ p'}$ we have
\begin{align*}
&\int_{\bR^d}\left(-\tilde a^{ij}\tilde v_{x^i}\varphi_{x^j}
-\tilde a^j\tilde v \varphi_{x^j}
+\tilde b^i \tilde v_{x^i} \varphi+(\tilde c-\lambda)\tilde v \varphi\right)\,dx\\
&=\int_{\bR^d_+}\left(-a^{ij}v_{x^i}\phi_{x^j}
- a^j v \phi_{x^j}
+ b^i  v_{x^i} \phi+(  c-\lambda) v \phi\right)\,dx,
\end{align*}
where
$\phi(x):=\varphi(x)-\varphi(-x^1,x)$. It is clear that $\phi\in W^1_{p'}(\bR_+^d)$ and has zero trace on $\partial\bR_+^d$. By \eqref{eq3.2.36}, the integral above is equal to
$$
\int_{\bR_+^d}-g_j\phi_{x^j}+f\phi\,dx=\int_{\bR^d}-\tilde g_j\varphi_{x^j}+\tilde f\varphi\,dx,
$$
which implies that $\tilde v\in W^1_p$ is a solution to \eqref{eq24.3.29}. By the uniqueness, we get $u=\tilde v$, which implies that $u=v$ in $\bR^d_+$. Finally, the estimate \eqref{eq24.2.52pm} follows from \eqref{eq080904} and Lemma \ref{lem8.2}. The theorem is proved.

\end{proof}

\begin{proof}[Proof of Theorem \ref{thm6.7}]
We define $\tilde a^{ij}$, $\tilde a^i$, $\tilde {b}^i$ and $\tilde c$ as in the proof of Theorem \ref{thm6.6}. Let $\tilde L$ be the divergence form elliptic operator with coefficients $\tilde a^{ij},\tilde a^i,\tilde {b}^i,\tilde c$. Different from above, we define
$$
\tilde f(x)=f(|x^1|,x'),
$$
$$
\tilde g_1(x)={\text{sgn}(x^1)}g_1(|x^1|,x'),\quad \tilde g_j(x)=g_j(|x^1|,x'),\,\,j\ge 2.
$$
Recall that $\tilde a^{ij}$ satisfy Assumption \ref{assumption20080424} ($2\gamma$). Moreover, we have $\tilde f$, $\tilde g\in L_{ p}(\bR^d)$. Due to Theorem \ref{th081901}, we can find $\gamma>0$ and $\lambda_0\ge 0$ such that there exists a unique solution $u\in W^1_{ p}(\bR^d)$ solving \eqref{eq24.3.29}
provided that $\lambda> \lambda_0$. By the definition of the coefficients and the data, we have
$$
\tilde Lu(-x^1,x')-\lambda u(-x^1,x')=\Div \tilde g(x)+\tilde f(x)\quad \text{in}\,\,\bR^d.
$$
Consequently, $u(-x^1,x')$ is also a solution to \eqref{eq24.3.29}. By the uniqueness of the solution, we obtain $u(x)=u(-x^1,x')$.

Let $p'$ be such that $1/p+1/p'=1$. For any $\phi\in W^1_{ p'}(\bR^d_+)$, denote $\tilde \phi$ to be its even extension with respect to $x^1$. Since $u$ satisfies \eqref{eq24.3.29}, integrating by parts gives
\begin{multline}\label{eq27.4.50pm}
\int_{\bR^d}\left(-\tilde a^{ij}u_{x^i}\tilde\phi_{x^j}
-\tilde a^ju \tilde\phi_{x^j}
+\tilde b^i u_{x^i}\tilde\phi+(\tilde c-\lambda)u\tilde\phi\right)\,dx\\
=\int_{\bR^d}\left(-\tilde g_j\tilde\phi_{x^j}
+\tilde f\tilde\phi\right)\,dx.
\end{multline}
By the definition of $\tilde a^{ij}$, $\tilde a^i$, $\tilde {b}^i$, $\tilde c$, $\tilde g$,  and $\tilde f$  as well as the evenness of $u$ and $\tilde \phi$, all terms inside the integrals  in  \eqref{eq27.4.50pm} are even with respect to $x^1$. Thus, \eqref{eq27.4.50pm} implies
\begin{equation}
                                                \label{eq27.4.56pm}
\int_{\bR_+^d}-a^{ij}u_{x^i}\phi_{x^j}-a^ju \phi_{x^j}+b^i u_{x^i}\phi+(c-\lambda)u\phi\,dx=\int_{\bR_+^d}-g_j\phi_{x^j}+f\phi\,dx.
\end{equation}
Since $\phi\in W^1_{ p'}(\bR^d_+)$ is arbitrary, by the definition of weak solutions, $u$ solves \eqref{eq2008061902}. This proves the existence of the solution.

For the uniqueness, let $v$ be another solution of \eqref{eq2008061902} so that, for any $\phi\in W^1_{ p'}(\bR^d_+)$,  the equality  \eqref{eq27.4.56pm} holds. Let $\tilde v$ to be the odd extension of $v$ with respect to $x^1$. Then by the definition of $\tilde a^{ij}$, $\tilde a^i$, $\tilde {b}^i$, $\tilde c$, $\tilde g$,  and $\tilde f$, for any $\varphi\in W^1_{ p'}$ we have
\begin{align*}
&\int_{\bR^d}\left(-\tilde a^{ij}\tilde v_{x^i}\varphi_{x^j}
-\tilde a^j\tilde v \varphi_{x^j}
+\tilde b^i \tilde v_{x^i} \varphi+(\tilde c-\lambda)\tilde v \varphi\right)\,dx\\
&=\int_{\bR^d_+}\left(-a^{ij}v_{x^i}\varphi_{x^j}
- a^j v \varphi_{x^j}
+ b^i  v_{x^i} \varphi+(  c-\lambda) v \varphi\right)\,dx\\
&\,\,+\int_{\bR^d_+}\left(-a^{ij}v_{x^i}\varphi_{x^j}(-x^1,x')
- a^j v \varphi_{x^j}(-x^1,x')
+ b^i  v_{x^i} \varphi(-x^1,x')\right)\\
&\,\,+\int_{\bR^d_+}(c-\lambda) v \varphi(-x^1,x')\,dx.
\end{align*}
Due to \eqref{eq27.3.58pm}, the sum above is equal to
\begin{align*}
&\int_{\bR_+^d}-g_j\varphi_{x^j}+f\varphi\,dx
+\int_{\bR_+^d}-g_j\varphi_{x^j}(-x^1,x')+f\varphi(-x^1,x')\,dx\\
&=\int_{\bR^d}-\tilde g_j\varphi_{x^j}+\tilde f\varphi\,dx.
\end{align*}
This yields that $\tilde v\in W^{1}_{p}$ is  a solution of \eqref{eq24.3.29}. By the uniqueness, we get $u=\tilde v$, which implies that $u=v$ in $\bR^d_+$. Finally, the estimate \eqref{eq27.3.49pm} follows from \eqref{eq080904} and Lemma \ref{lem8.2}. The theorem is proved.
\end{proof}

\mysection{Equations in Lipschitz domains}
                                            \label{sec7}
In this section we present the proofs of Theorem \ref{thmA} and \ref{thmB}. Recall that we not only assume the leading coefficients $a^{ij}$ are partially small BMO, but also assume that they have small BMO semi-norms  in some neighborhood of $\partial \Omega$.
First we have the following classical $W^1_2$-solvability of the Dirichlet problem
\begin{equation}	\label{eq27.9.45}
\left\{
  \begin{array}{ll}
    \cL u=\Div g+f & \hbox{in $\Omega$} \\
    u=0 & \hbox{on $\partial \Omega$}
  \end{array};
\right.
\end{equation}
see, for example, \cite{Gilbarg&Trudinger:book:1983}.
\begin{theorem}
                                                \label{thm9.1}
Let $\Omega$ be a bounded domain. Assume $a^i_{x^i}+c\le 0$ in $\Omega$ in the weak sense. Then for any $f$, $g = (g_1, \cdots, g_d) \in L_{2}(\Omega)$ there exists a unique $u\in W^1_2(\Omega)$ solving \eqref{eq27.9.45}.
Moreover, we have
\begin{equation}
                            \label{eq27.9.46}
\|u\|_{W^1_2(\Omega)}\le N\|f\|_{L_2(\Omega)}+N\|g\|_{L_2(\Omega)}.
\end{equation}
\end{theorem}
In the sequel, we only focus on the case $p\in (2,\infty)$, since the remaining case $p\in (1,2)$ follows immediately from the duality. Because $\Omega$ is bounded, under the conditions of Theorem \ref{thmA}, we have $f,g\in L_p(\Omega)\subset L_2(\Omega)$. Owing to Theorem \ref{thm9.1}, there is a unique solution $u\in W^1_2(\Omega)$ to \eqref{eq27.9.45}. As is well known, by the method of continuity, in order to prove Theorem \ref{thmA} it suffices to show the a priori estimate \eqref{eq27.8.52} for $u\in W^1_p(\Omega)$. We need the following local estimates.

\begin{lemma}
                                            \label{lem9.2}
Let $\Omega'\Subset\Omega$, $f,g=(g_1,\cdots,g_d)\in L_p(\Omega)$, \and $\lambda_0$ and $\gamma$ are constants taken from Theorem \ref{th081901}. Then under Assumption \ref{assumption20080424} ($\gamma$),
for any $u \in W_p^1$, we have
\begin{equation}
                                    \label{eq28.2.59}
\sqrt{\lambda} \| u_x \|_{L_p(\Omega')}
\le N \left(\sqrt{\lambda} \| g \|_{L_p(\Omega)} +  \| f \|_{L_p(\Omega)}+ \lambda \| u \|_{L_p(\Omega)}+ \| u_x \|_{L_p(\Omega)}\right),
\end{equation}
provided that $\lambda\ge \lambda_0$ and
\begin{equation*}							 
\cL u = \Div g + f\quad \text{in}\,\,\Omega,
\end{equation*}
where $N=N(d,p,\delta,K,R_0,\Omega',\Omega)>0$.
\end{lemma}
\begin{proof}
Fix a $\lambda\ge \lambda_0$. We take a smooth cut-off function $\eta\in C_0^\infty(\Omega)$ such that $\eta\equiv 1$ in $\Omega'$. It is easily seen that
\begin{equation*}
\cL (\eta u)-\lambda \eta u=\Div \tilde g+\tilde f\quad \text{in}\,\,\bR^d,
\end{equation*}
where
\begin{equation}
                                    \label{Eq28.2.46}
\tilde g_j=\eta g_j+a^{ij}\eta_{x^i}u,\quad \tilde f=\eta f-\eta_{x^i}g_i+a^{ij}u_{x^i}\eta_{x^j}+a^i u \eta_{x^i}+b^i u \eta_{x^i}-\lambda \eta u.
\end{equation}
Due to Theorem \ref{th081901} (i), we have
\begin{equation}
                                            \label{eq28.2.47}
\sqrt{\lambda} \| u_x \|_{L_p(\Omega')}\le
\sqrt{\lambda} \| (\eta u)_x \|_{L_p}
\le N\sqrt \lambda \|\tilde g\|_{L_p}+N\|\tilde f\|_{L_p}.
\end{equation}
By \eqref{Eq28.2.46}, the right-hand side of \eqref{eq28.2.47} is less than the right-hand side of \eqref{eq28.2.59}. The lemma is proved.
\end{proof}

For $r>0$, we denote $B_r^+=B_r\cap \bR^d_+$.
\begin{lemma}
                                                \label{lem9.3}
Let $0<r<R<\infty$, $f,g=(g_1,\cdots,g_d)\in L_p(B_R^+)$, \and $\lambda_0$ and $\gamma$ are constants taken from Theorem  \ref{thm6.6}. Then under Assumption \ref{assumption20080424} ($\gamma$), for any $u\in W_p^1(B_R^+)$ satisfying $u=0$ on $B_R\cap \partial \bR^d_+$, we have
\begin{equation}
                                            \label{eq28.3.00}
\sqrt{\lambda} \| u_x \|_{L_p(B_r^+)}
\le N (\sqrt{\lambda} \| g \|_{L_p(B_R^+)} + \| f \|_{L_p(B_R^+)}
+ \lambda \| u \|_{L_p(B_R^+)}+\| u_x \|_{L_p(B_R^+)}),
\end{equation}
provided that $\lambda\ge \lambda_0$ and
\begin{equation*}							 
\cL u = \Div g + f\quad \text{in}\,\,B_R^+,
\end{equation*}
where $N=N(d,p,\delta,K,R_0,r,R)>0$.
\end{lemma}
\begin{proof}
The lemma follows immediately from the proof of Theorem \ref{thm6.6} and Lemma \ref{lem9.2}. We omit the detail.
\end{proof}
\begin{remark}
By an iteration argument, one actually can drop the $\|u_x\|_{L_p}$ term on the right-hand side of \eqref{eq28.2.59} and \eqref{eq28.3.00}. However, we will not use this in our proof.
\end{remark}

Next we locally flatten the boundary of $\partial\Omega$ under Assumption \ref{assump1} ($\gamma$) and \ref{assump2} ($\theta$). Let us choose a point $x_0\in \partial \Omega$ and a number $r_0=\min\{R_1,R_2\}$, so that
$$
\Omega\cap B_{r_0}(x_0) = \{x \in  B_{r_0}(x_0)\, :\, x^1 >\phi(x')\}.
$$
We define
$$
y_1=x_1-\phi(x'):=\Phi^1(x),\quad y^j=x^j:=\Phi^j(x),\,\,j\ge 2.
$$
There exists  a  small $r_1>0$ depending on $r_0$ such that
$$
B_{r_1}\subset \Phi(B_{r_0}(x_0)),\quad B_{r_1}^+\subset \Phi(\Omega\cap B_{r_0}(x_0)),
$$
where we assumed that, without loss of generality, $0 = y_0 = \Phi(x_0)$.
Denote $v(y)=u(\Psi(y))$ for any $y\in B_{r_1}^+$, where $\Psi=\Phi^{-1}$.
If $u \in W_p^1(\Omega)$ satisfies the equation \eqref{eq27.9.45}, it is easily seen that $v$ satisfies $v=0$ on    $B_{r_1}\cap \partial \bR^d_+$ and
$$
\hat \cL v=\Div \hat g+\hat f\quad\text{in}\,\, B_{r_1}^+,
$$
where for $y\in B_{r_1}$,
$$
\hat a^{ij}(y)=a^{kl}(\Psi(y))\Phi^i_{x^k}(\Psi(y))\Phi^j_{x^l}(\Psi(y)),\quad
\hat a^{i}(y)=a^{k}(\Psi(y))\Phi^i_{x^k}(\Psi(y)),
$$
$$
\hat b^{i}(y)=b^{k}(\Psi(y))\Phi^i_{x^k}(\Psi(y)),\quad \hat c(y)=c(\Psi(y)),
$$
$$
\hat f^i(y)=f^{k}(\Psi(y))\Phi^i_{x^k}(\Psi(y)),\quad \hat g(y)=g(\Psi(y)).
$$
These coefficients satisfy the boundedness and ellipticity conditions with possibly different but comparable constants. Following the argument in \cite{Byun05a}, we know that $\hat a^{ij}$ satisfy
\begin{equation}
                                                    \label{eq28.5.02}
\sup_{ij}\dashint_{B_r(x_1)}|\hat a^{ij}(x)-(\hat a^{ij})_{B_r(x_1)}|\,dx\le N_0(\theta+\gamma),
\end{equation}
for any $x_1\in B_{r_1/2}$ and $r\in (0,r_1/2]$,
where $N_0$ is independent of $\theta$ and $\gamma$. We may change the values of $\hat a^{ij}$ outside $B_{r_1/8}$ and extend them to $\bR^d$  so that the boundedness and ellipticity condition are still satisfied with the same constants, and \eqref{eq28.5.02} is satisfied for any $x_1\in \bR^d$ and $r\in (0,r_2]$ with a possibly larger $N_0$ and $r_2=\min\{1/4,\theta+\gamma\}r_1$. Indeed, this can be done by considering
$$
\hat a^{ij}\eta(\cdot/r_1)+\delta^{ij}(1-\eta(\cdot/r_1)),
$$
where $\eta\in C_0^\infty(B_{1/4})$ satisfying $\eta=1$ in $B_{1/8}$.
Now we choose sufficiently small $\theta$ and $\gamma$ such that $N_0(\theta+\gamma)$ is less than the constant  $\gamma$ in Lemma \ref{lem9.3}. By Lemma \ref{lem9.3}, we get
\begin{multline*}
\sqrt{\lambda} \| v_x \|_{L_p(B_{r_1/16}^+)}\\
\le N (\sqrt{\lambda} \| \hat g \|_{L_p(B_{r_1/8}^+)} + \| \hat f \|_{L_p(B_{r_1/8}^+)}
+ \lambda \| v \|_{L_p(B_{r_1/8}^+)}+\| v_x \|_{L_p(B_{r_1/8}^+)}),
\end{multline*}
which implies
\begin{align}
\sqrt{\lambda} \|u_x \|_{L_p(\Omega\cap B_{r_3})}
\le N ( \sqrt{\lambda} \| g \|_{L_p(\Omega\cap B_{r_0})} + \| f \|_{L_p(\Omega\cap B_{r_0})}\nonumber\\
                                                \label{eq5.29pm}
+ \lambda \|u\|_{L_p(\Omega\cap B_{r_0})}+\| u_x\|_{L_p(\Omega\cap B_{r_0})}),
\end{align}
for $\lambda\ge \lambda_0$, where $r_3$ depends only on  $r_0$.

Now we are in the position to prove Theorem \ref{thmA}
\begin{proof}[Proof of Theorem \ref{thmA}]
By using a partition of unity, we get from \eqref{eq28.2.59} and \eqref{eq5.29pm} that
$$
\sqrt{\lambda} \| u_x \|_{L_p(\Omega)}
\le N_2 (\sqrt{\lambda} \| g \|_{L_p(\Omega)} +  \| f \|_{L_p(\Omega)}+ \lambda \| u \|_{L_p(\Omega)}+ \| u_x \|_{L_p(\Omega)})
$$
for $\lambda\ge \lambda_0$. To absorb the $\|u_x\|_{L_p(\Omega)}$ term on the right-hand side, we take and fix a sufficiently large $\lambda$ so that $N_2\le \sqrt \lambda/2$. Therefore,
\begin{equation*}
\| u_x \|_{L_p(\Omega)}
\le N_3 (\| g \|_{L_p(\Omega)} +  \| f \|_{L_p(\Omega)}+\| u \|_{L_p(\Omega)}).
\end{equation*}
Take $p_1\in (p,\infty)$ such that  $1-d/p>-d/p_1$. By H\"older's inequality, Young's inequality and Poincar\'e-Sobolev inequality, we get for any $\epsilon>0$,
$$
\|u\|_{L_p(\Omega)}\le N(\epsilon)\|u\|_{L_2(\Omega)}+ \epsilon\|u\|_{L_{p_1}(\Omega)}\le N(\epsilon)\|u\|_{L_2(\Omega)}+N_4\epsilon\|u_x\|_{L_{p}(\Omega)}.
$$
Choosing $\epsilon=1/(2N_3N_4)$ and using \eqref{eq27.9.46}, we obtain \eqref{eq27.8.52}. The theorem is proved.
\end{proof}

Next we turn to study the conormal derivative problem:
\begin{equation}	\label{eq29.11.07}
\left\{
  \begin{array}{ll}
    \cL u=\Div g+f & \hbox{in $\Omega$} \\
    a^{ij}u_{x^i}n^j+a^j u n^j=g_j n^j & \hbox{on $\partial \Omega$}
  \end{array}.
\right.
\end{equation}

\begin{lemma}
                                            \label{lem9.5}
Let $\Omega$ be a bounded domain and $u\in W^1_2(\Omega)$. Assume in the weak sense $a^i_{x^i}+c\le 0$ in $\Omega$, $\cL u=0$ in $\Omega$, and $a^{ij}u_{x^i}n^j+a^j u n^j=0$ on $\partial \Omega$.

\noindent (i) If in the weak sense $a^i_{x^i}+c\equiv 0$ in $\Omega$ and $a^i n^i=0$ on $\partial \Omega$, then we have $u\equiv C$ for a constant $C\in \bR$.

\noindent (ii) Otherwise, we have $u\equiv 0$.
\end{lemma}
\begin{proof}
Owing to the strong maximum principle for the conormal derivative problem (cf. \cite{Gilbarg&Trudinger:book:1983} and \cite{Lieb}), under the assumption of the lemma, $u$ is a constant in $\Omega$. Now assertions (i) and (ii) follow from the definitions \eqref{eq27.9.28pm} and \eqref{eq0831}.
\end{proof}

Owing to Lemma \ref{lem9.5}, we get the $W^1_2$-solvability for the conormal derivative problem.
\begin{theorem}
                                                \label{thm9.6}
Let $\Omega$ be a bounded domain. Assume $a^i_{x^i}+c\le 0$ in $\Omega$ in the weak sense.

\noindent (i) If in the weak sense $a^i_{x^i}+c\equiv 0$ in $\Omega$ and $a^i n^i=0$ on $\partial \Omega$, for any $f$, $g = (g_1, \cdots, g_d) \in L_{2}(\Omega)$ there exists a unique up to a constant $u\in W^1_2(\Omega)$ solving \eqref{eq29.11.07} provided that
\begin{equation}
                                            \label{eq11.00}
b^i=c=0,\quad \int_\Omega f\,dx=0.
\end{equation}
Moreover, we have
\begin{equation}
                            \label{eq29.11.25}
\|u_x\|_{L_2(\Omega)}\le N\|f\|_{L_2(\Omega)}+N\|g\|_{L_2(\Omega)}.
\end{equation}

\noindent (ii) Otherwise, the solution is unique and we have
\begin{equation*}
\|u\|_{W^1_2(\Omega)}\le N\|f\|_{L_2(\Omega)}+N\|g\|_{L_2(\Omega)}.
\end{equation*}
\end{theorem}
\begin{proof}
We remark that for equations without lower order terms and $f$, this result was proved in \cite{ByunWang05}.

First we have the unique solvability in $W^1_2(\Omega)$ of $\cL u-\lambda u=\Div g+f$ with the same boundary condition for a sufficiently large $\lambda$. Indeed this follows from the coercivity of the bilinear form
$$
\mathcal B(u,v):=\int_{\Omega}\left(a^{ij}u_{x^i}v_{x^j}+a^ju v_{x^j}-b^i u_{x^i}v+(\lambda-c)uv\right)\,dx
$$
for large enough $\lambda$, and the Lax-Milgram lemma. We fix such a $\lambda$ and denote the corresponding resolvent operator to be $\cL_\lambda^{-1}$, which is bounded from $L_2(\Omega)$ to $W^1_2(\Omega)$.

{\em Part (i):}
Let
$$
H^1(\Omega)=\{v\in W^1_2(\Omega)\,|\,(v)_{\Omega}=0\},
$$
and $\cI:H^1(\Omega)\to L_2(\Omega)$ be the natural compact imbedding. Also we define $\cT:H^1(\Omega)\to H^1(\Omega)$ by
$$
\cT v=\cL_\lambda^{-1}\cI v-(\cL_\lambda^{-1}\cI v)_\Omega.
$$
Clearly $\cT$ is a compact operator on $H^1(\Omega)$.
Bearing Lemma \ref{lem9.5} in mind, to prove part (i), it suffices to show the unique solvability of \eqref{eq29.11.07} in $H^1(\Omega)$ and the bound \eqref{eq29.11.25}.
We claim that if $u$ satisfies
\begin{equation}
 				\label{eq23.10.30}
u+\lambda \cT u=\cL_\lambda^{-1}(\Div g+f)-\left(\cL_\lambda^{-1}(\Div g+f)\right)_\Omega,
\end{equation} then $u$ also solves \eqref{eq29.11.07}. Indeed, clearly $u$ solves
\begin{equation}	\label{eq29.11.078}
\left\{
  \begin{array}{ll}
    \cL u=\Div g+f+C & \hbox{in $\Omega$} \\
    a^{ij}u_{x^i}n^j+a^j u n^j=g_j n^j & \hbox{on $\partial \Omega$}
  \end{array},
\right.
\end{equation}
for some constant $C$. To see $C=0$, we take $\phi\equiv 1$ in the integral formulation of \eqref{eq29.11.078} and use \eqref{eq11.00}.

Therefore, any solution to
\begin{equation}
 					\label{eq23.10.35}
u+\lambda \cT u=0
\end{equation}
also solves \eqref{eq29.11.07} with $g^j\equiv f\equiv 0$. Due to Lemma \ref{lem9.5} (i) \eqref{eq23.10.35} has a unique solution $u\equiv 0$ in $H^1(\Omega)$. This together with the Fredholm alternative shows that there is a unique $u\in H^1(\Omega)$ satisfying \eqref{eq23.10.30}.
Moreover, we have the estimate
$$
\|u_x\|_{L_2(\Omega)}\le N\|f\|_{L_2(\Omega)}+N\|g\|_{L_2(\Omega)}.
$$
Part (i) of the theorem is proved.

{\em Part (ii):}
The proof is similar. Instead of the space $H^1(\Omega)$, we solve the equation in the usual Sobolev space $W^1_2(\Omega)$. Let $\cI:W_2^1(\Omega)\to L_2(\Omega)$ be the natural compact imbedding. Also we define $\cT:W_2^1(\Omega)\to W_2^1(\Omega)$ by
$$
\cT v=\cL_\lambda^{-1}\cI v.
$$
Clearly $\cT$ is a compact operator on $W_2^1(\Omega)$. Notice that for any $u\in W_2^1(\Omega)$ \eqref{eq29.11.07} is equivalent to
$$
u+\lambda \cT u=\cL_\lambda^{-1}(\Div g+f).
$$
By the Fredholm alternative, the unique solvability in $W_2^1(\Omega)$ and the bound follow from the uniqueness in $W_2^1(\Omega)$ of the trivial solution to $u+\lambda \cT u=0$, which is due to Lemma \ref{lem9.5} (ii). This completes the proof of the theorem.
\end{proof}

We also need a boundary estimate analogue to Lemma \ref{lem9.3}.
\begin{lemma}
                                                \label{lem9.7}
Let $0<r<R<\infty$, $f,g=(g_1,\cdots,g_d)\in L_p(B_R^+)$, \and $\lambda_0$ and $\gamma$ are constants taken from Theorem \ref{thm6.7}. Then under Assumption \ref{assumption20080424} ($\gamma$), for any $u\in W_p^1(B_R^+)$   we have
\begin{equation*}
\sqrt{\lambda} \| u_x \|_{L_p(B_r^+)}
\le N ( \sqrt{\lambda} \| g \|_{L_p(B_R^+)} + \| f \|_{L_p(B_R^+)}
+ \lambda \| u \|_{L_p(B_R^+)}+\| u_x \|_{L_p(B_R^+)}),
\end{equation*}
provided that $\lambda\ge \lambda_0$ and
\begin{equation*}	
\left\{
  \begin{array}{ll}
    \cL u=\Div g+f & \hbox{in $B_R^+$} \\
    a^{i1}u_{x^i}+a^1 u=g_1 & \hbox{on $B_R\cap \partial \bR^d_+$}
  \end{array},
\right.
\end{equation*}
where $N=N(d,p,\delta,K,R_0,r,R)>0$.
\end{lemma}
\begin{proof}
The lemma follows immediately from the proof of Theorem \ref{thm6.7} and Lemma \ref{lem9.2}.
\end{proof}
\begin{proof}[Proof of Theorem \ref{thmB}]
Notice that under the conditions of Theorem \ref{thmB}, we still have \eqref{eq5.29pm} by relying on Lemma \ref{lem9.7} instead of Lemma \ref{lem9.3}. This together with Lemma \ref{lem9.2} yields the conclusions of the theorem by the same reasoning as in the proof of Theorem \ref{thmA}.
\end{proof}

\mysection{Auxiliary results for the mixed norm case}
                                            \label{sec8}

The results in this section are similar to those in \cite{Krylov_2007_mixed_VMO}
(specifically, Lemma 8.2, Corollary 8.3 and Corollary 8.4 in \cite{Krylov_2007_mixed_VMO}).
However,  since the conditions on the operators considered here are more general than those in \cite{Krylov_2007_mixed_VMO}, it is not possible to refer to the results in \cite{Krylov_2007_mixed_VMO}
based on the idea that the elliptic case can be considered as the time independent parabolic case.
In addition, contrary to the parabolic case where the Cauchy problem with zero initial condition is considered,
we are not able to use the solvability to the equation
$(a^{ij}u_{x^i})_{x^j} = \Div g$ in the whole space.
Because of these differences,
we present here complete proofs.
Throughout the section, set
$$
\cL_0 u = \left( a^{ij} u_{x^i} \right)_{x^j}
$$

\begin{lemma}							\label{lem080904}
Let $r \in (0,R_0]$, $\kappa > 1$, and $q \in (1,p]$, $p \in (1, \infty)$.
Assume that
\begin{equation}							 \label{eq081802}
\frac{1}{q} - \frac{1}{p} \le \frac{1}{d}.
\end{equation}
Then there exists $\gamma=\gamma(d,p,q,\delta,K)>0$ such that, under Assumption \ref{assumption20080424} ($\gamma$),
if $u \in W_{q,\text{loc}}^1$
satisfies $ \cL_0 u = 0$ in $B_{\kappa r}$,
then $u$, $u_x \in L_p(B_r)$
and
\begin{equation}							 \label{eq081901}
r^{-1}\left( |u|^p \right)_{B_r}^{1/p}
+ \left( |u_x|^p \right)_{B_r}^{1/p}
\le N \left( |u_x|^q + r^{-q} |u|^q \right)_{B_{\kappa r}}^{1/q},
\end{equation}
where $N = N(d, p, q, \delta, K, \kappa)$.
\end{lemma}

\begin{proof}
First we show that it is enough to prove \eqref{eq081901} only for $r = R_0=1$.
To do this, assume that \eqref{eq081901} holds true for $r = R_0= 1$ and
let $u$ be a function in $W_{q,\text{loc}}^1$ such that $ \cL_0 u = 0$ in $B_{\kappa r}$.
Then $\hat{u}(x) := u(rx)$ satisfies
$$
(\hat{a}^{ij}(x) \hat{u}_{x^i})_{x^j}  = 0
$$
in $B_{\kappa}$, where $\hat{a}^{ij}(x) = a^{ij}(rx)$.
The coefficients $\hat{a}^{ij}$ carry the same constants
$\delta$ and $K$ as $a^{ij}$. Moreover, $\hat{a}^{ij}$ satisfy Assumption \ref{assumption20080424} ($\gamma$) with $R_0$ replaced by $1$ because
$$
\text{osc}_{x'}\left(\hat{a}^{ij},\Gamma_{\rho}(x)\right)
= \text{osc}_{x'}\left(a^{ij},\Gamma_{r \rho}(rx)\right),
$$
which implies that $\hat{a}^{\#}_{1} = a^{\#}_{r}\le  a^{\#}_{R_0} \le \gamma$.
Then by applying the estimate \eqref{eq081901} with $r =R_0= 1$ to the equation
$ \hat{\cL}_0 \hat{u} = 0$ in $B_{\kappa}$,
we have
$$
\left( |\hat{u}|^p \right)_{B_1}^{1/p}+
\left( |\hat{u}_x|^p \right)_{B_1}^{1/p}
\le N \left( |\hat{u}_x|^q + |\hat{u}|^q \right)_{B_{\kappa}}^{1/q}
$$
with the same constant $N$.
Returning back to $u$ proves the lemma for $r \in (0,R_0]$.

To deal with the case $r = R_0=1$,
we fix $\lambda > \lambda_0$ and $\gamma$, where $\lambda_0$ and $\gamma$ are from Theorem \ref{th081901} which work for both $p$ and $q$. First it follows from the Sobolev imbedding theorem that

\begin{equation}							 \label{eq081906}
\left( |u|^p \right)^{1/p}_{B_{\kappa}}
\le N \left( |u|^q + |u_x|^q \right)^{1/q}_{B_{\kappa}},
\end{equation}
where $N = N(d,p,q,\kappa)$. In particular, this shows that $\left( |u|^p \right)^{1/p}_{B_1}$
is controlled by the right-hand side of \eqref{eq081901}.

Let $\eta \in C_0^{\infty}$ be
such that $\eta = 1$ on $B_{1}$ and $\eta = 0$ outside $B_{\kappa}$.
Then
$$
( \cL_0 - \lambda) (\eta u) = \Div g + f,
$$
where
\begin{equation}							 \label{eq081908}
g_j = \sum_{i=1}^d a^{ij} u \eta_{x^i},
\quad
f = \sum_{i,j=1}^d a^{ij} u_{x^i} \eta_{x^j} - \lambda \eta u.
\end{equation}
Since $g$ has a compact support in $B_{\kappa}$,
the inequality \eqref{eq081906} implies
\begin{equation}							 \label{eq081904}
\| g \|_{L_p}
\le N \left( |u|^q + |u_x|^q \right)^{1/q}_{B_{\kappa}}.
\end{equation}
Now using the fact that $f \in L_q$
and the well-known $L_p$-theory for the Laplace operator,
we find a unique solution $w \in W_q^2$ to the equation
$$
\Delta w - \lambda w = f.
$$
By the Sobolev embedding theorem again
and the $L_q$-estimate corresponding to the above equation,
\begin{equation}							 \label{eq081905}
\| w \|_{L_p} + \| w_x \|_{L_p}
\le N \| w \|_{W_q^2}
\le N \|f\|_{L_q}
\le N \left( |u|^q + |u_x|^q \right)^{1/q}_{B_{\kappa}}.
\end{equation}
Define $v := \eta u - w$,
which is in $W_q^1$
because $\eta u$, $w \in W_q^1$.
In addition,
\begin{equation}							 \label{eq081903}
\cL_0 v - \lambda v = \Div g - ((a^{ij} - \delta^{ij}) w_{x^i})_{x^j}.
\end{equation}
Note that $g$, $w_x \in L_q$.
Thus, by Theorem \ref{th081901},
$v$ is the unique solution in $W_q^1$ to the above equation.
On the other hand,
by \eqref{eq081904} and \eqref{eq081905}
we have $g$, $w_{x^i} \in L_p$.
Thus, again by Theorem \ref{th081901},
there exists a unique solution in $W_p^1$ to the equation
\eqref{eq081903}.
This implies that $v$ has to be the unique solution in $W_p^1$ to the equation \eqref{eq081903}.
Moreover,
$$
\| v_x \|_{L_p} \le N \left(\| g \|_{L_p} + \| w_x \|_{L_p}\right).
$$
From this, \eqref{eq081904}, \eqref{eq081905},
and $\eta u = w + v$,
$$
\left( |u_x|^p \right)^{1/p}_{B_1}
\le N \| (\eta u)_x \|_{L_p}
\le N \left( |u|^q  + |u_x|^q \right)_{B_{\kappa}}^{1/q}.
$$
This finishes the proof.
\end{proof}

Assume that $1 < q < p$.
Then we can always find $p_0, p_1, \cdots, p_m$
such that $p_0 = q$, $p_m = p$, and $1/p_i - 1/p_{i+1} \le 1/d$, $i = 0, \cdots, m-1$.
Using the above lemma as many times as needed, we prove

\begin{corollary}							 \label{cor082001}
Let $r \in (0,R_0]$, $\kappa > 1$, and $q \in (1,p]$, $p \in (1, \infty)$.
Assume that $u \in W_{q,\text{loc}}^1$
satisfies $ \cL_0 u = 0$ in $B_{\kappa r}$.
Then there exists $\gamma=\gamma(d,p,q,\delta,K)>0$ such that, under Assumption \ref{assumption20080424} ($\gamma$), we have $u$, $u_x \in L_p(B_r)$
and
$$
r^{-1}\left( |u|^p \right)_{B_r}^{1/p} + \left( |u_x|^p \right)_{B_r}^{1/p}
\le N \left( |u_x|^q + r^{-q} |u|^q \right)_{B_{\kappa r}}^{1/q},
$$
where $N = N(d, p, q, \delta, K, \kappa)$.
\end{corollary}

As noted earlier, the corollary below corresponds to Corollary 8.4 in \cite{Krylov_2007_mixed_VMO},
but the statement is a little different.

\begin{corollary}							 \label{cor082002}
Let $p, q \in (1,\infty)$, $\lambda \ge 0$, and $0 < r \le R_0/\sqrt{2}$.
Assume that $u \in W_{q,\text{loc}}^1$
and $ \cL_0 u - \lambda u = 0$ in $B_{2r}$.
Then there exists $\gamma=\gamma(d,p,q,\delta,K)>0$ such that, under Assumption \ref{assumption20080424} ($\gamma$), we have $u_x \in L_p(B_r)$ and
\begin{equation}							 \label{eq082002}
\left( |u_x|^p \right)^{1/p}_{B_r}
\le N \left( |u_x|^q + \lambda^{q/2} |u|^q \right)^{1/q}_{B_{2r}},
\end{equation}
where $N = N(d,p,q,\delta, K)$.
\end{corollary}

\begin{proof}
It suffices to prove the case $q < p$. We take $\gamma$ from Corollary \ref{cor082001}.
First assume that $\lambda = 0$.
Set $\tau = \sqrt{2}r \in (0,R_0]$ and consider $u - (u)_{B_{\sqrt{2}\tau}}$, which is in $W_{q,\text{loc}}^1$
and satisfies $ \cL_0 u = 0$ in $B_{\sqrt{2}\tau}$.
Then by Corollary \ref{cor082001} and the Poincar\'e inequality
\begin{equation*}							 
\left( |u_x|^p \right)_{B_{\tau}}^{1/p}
\le N \left( |u_x|^q + \tau^{-q} |u - (u)_{B_{\sqrt{2}\tau}}|^q \right)_{B_{\sqrt{2}\tau}}^{1/q}\le N \left( |u_x|^q \right)_{B_{\sqrt{2}\tau}}^{1/q}.
\end{equation*}
This proves \eqref{eq082002} when $\lambda = 0$.

If $\lambda > 0$, we set
$$
\hat{u}(z) = u(x) \cos (\sqrt{\lambda} y),
$$
where $z \in \bR^{d+1}$, $z = (x,y)$, $x \in \bR^d$, and $y \in \bR$.
Then
$$
(a^{ij} \hat{u}_{x^i})_{x^j} + \hat{u}_{yy} = 0
$$
in $\hat{B}_{\sqrt{2}\tau} = \{ |z| \le \sqrt{2}\tau : z \in \bR^{d+1} \}$,
$\tau = \sqrt{2} r \in (0, R_0]$.
By the proof above for $\lambda = 0$, we obtain
$$
\left( |\hat{u}_z|^p \right)^{1/p}_{\hat{B}_{\tau}}
\le N \left( |\hat{u}_z|^q \right)^{1/q}_{\hat{B}_{\sqrt{2}\tau}}.
$$
Note that there exists a small constant $N = N(p)$,
independent of $r$ and $\lambda$,
such that
$$
N \le \dashint_{-r}^{\,\,\,\,r} |\cos(\sqrt{\lambda}y)|^p \, dy.
$$
Thus
$$
\dashint_{B_r} |u_x(x) |^p \, dx
\le N \dashint_{-r}^{\,\,\,\,r} \dashint_{B_r} |u_x(x) \cos(\sqrt{\lambda}y)|^p \, dx \, dy
\le N \dashint_{\hat{B}_{\tau}} | \hat{u}_z |^p \, dz.
$$
Also we have
$$
\dashint_{\hat{B}_{\sqrt{2}\tau}} |\hat{u}_z|^q \, dz
\le N \dashint_{-2r}^{\,\,\,\,2r} \dashint_{B_{2r}} |\cos(\sqrt{\lambda} y) u_x(x)|^q \, dx \, dy
$$
$$
+ N \dashint_{-2r}^{\,\,\,\,2r} \dashint_{B_{2r}} |\sqrt{\lambda}\sin(\sqrt{\lambda} y) u(x)|^q \, dx \, dy
\le N \left( |u_x|^q \right)_{B_{2r}}
+ N \lambda^{q/2} \left( |u|^q \right)_{B_{2r}}.
$$
Using the above two sets of inequalities we complete the proof of \eqref{eq082002}.
\end{proof}

\mysection{Results for the mixed norm case}
                                            \label{sec9}

As introduced earlier, the mixed norm $L_{q,p}$ of $u$ means
$$
\|u\|_{L_{q,p}} = \| u \|_{L_q^{\sbx_2}L_p^{\sbx_1}(\bR^d)}
= \left(\int_{\bR^{d_2}} \left( \int_{\bR^{d_1}} |u(\sbx_1, \sbx_2)|^p \, d \sbx_1 \right)^{q/p} \, \sbx_2\right)^{1/q}.
$$
Throughout the section, by $\cL$ we mean the elliptic operator in \eqref{eq0617_02},
the coefficients of which have the same conditions as in Section \ref{sec082001}.

We state the main results concerning elliptic equations in Sobolev spaces with mixed norms.
The proof of the first main result is presented at the end of Section \ref{mixednormsec}.

\begin{theorem}
                                \label{thm6.5}
Let $p, q \in (1,\infty)$, and $f$, $g = (g_1, \cdots, g_d) \in L_{q,p}$.
Then there exists a constant $\gamma=\gamma(d,p,q,\delta,K)$
such that, under Assumption \ref{assumption20080424} ($\gamma$),
the following hold true.

\noindent (i) There exist constants $\lambda_1$ and $N$, depending only on
$d_1$, $d_2$, $p$, $q$, $\delta$, $K$, and $R_0$, such that
$$
\sqrt{\lambda}\|u_x\|_{L_{q,p}}
+\lambda \|u\|_{L_{q,p}}
\leq N\sqrt{\lambda}\|g\|_{L_{q,p}}
+N\|f\|_{L_{q,p}},
$$
provided that $u \in W^1_{q,p}$,
$\lambda \ge \lambda_1$, and
\begin{equation}
                                        \label{eq090808_01}
\cL u - \lambda u = \Div g + f.
\end{equation}

\noindent (ii) For any $\lambda > \lambda_1$, there exists a unique $u\in W^1_{q,p}$ satisfying \eqref{eq090808_01}.

\noindent (iii)
If $a^i = b^i = c = 0$ and $a^{ij} = a^{ij}(x^1)$, i.e., measurable functions of $x^1 \in \bR$ only
with no regularity assumptions,
then one can take $\lambda_1 = 0$.
\end{theorem}

The following two theorems are about the Dirichlet problem and the conormal derivative problem
on a half space when Sobolev spaces with mixed norms are considered.
Since their proofs are basically the same as those of Theorem \ref{thm6.6} and \ref{thm6.7}, that is, we use Theorem \ref{thm6.5}, Lemma \ref{lem8.2}, and odd/even extensions,
we only state the theorems.

\begin{theorem}
Let $p, q \in (1,\infty)$ and $f$, $g = (g_1, \cdots, g_d) \in L_{q,p}(\bR^{d}_{+})$.
Then there exists a constant $\gamma=\gamma(d,p,q,\delta,K)$
such that, under Assumption \ref{assumption20080424} ($\gamma$),
the following hold true. There exist constants $\lambda_1$ and $N$,  depending only on
$d_1$, $d_2$, $p$, $q$, $\delta$, $K$, and $R_0$, such that
$$
\sqrt{\lambda}\|u_x\|_{L_{q,p}(\bR^{d}_{+})}
+\lambda \|u\|_{L_{q,p}(\bR^{d}_{+})}
\leq N\sqrt{\lambda}\|g\|_{L_{q,p}(\bR^{d}_{+})}
+N\|f\|_{L_{q,p}(\bR^{d}_{+})},
$$
provided that $\lambda \ge \lambda_1$ and $u \in W^1_{q,p}(\bR^d_{+})$ satisfies $u(0,x')=0$ and
\begin{equation}	\label{eq090808_02}
\cL u-\lambda u=\Div g+f\quad \text{in}\,\,\bR^d_+.
\end{equation}
Moreover, for any $\lambda > \lambda_1$ and $g,f\in L_{q,p}(\bR^{d}_{+})$, there exists a unique $u\in W^1_{q,p}(\bR^{d}_{+})$ satisfying \eqref{eq090808_02} and $u(0,x') = 0$.
\end{theorem}

\begin{theorem}[Conormal derivative problem on a half space]
Let $p$, $q \in (1,\infty)$ and $f$, $g = (g_1, \cdots, g_d) \in L_{q,p}(\bR^{d}_{+})$.
Then there exists a constant $\gamma=\gamma(d,p,q,\delta,K)$
such that, under Assumption \ref{assumption20080424} ($\gamma$),
the following hold true.
There exist constants $\lambda_1$ and $N$,  depending only on
$d_1$, $d_2$, $p$, $q$, $\delta$, $K$, and $R_0$, such that
$$
\sqrt{\lambda}\|u_x\|_{L_{q,p}(\bR^{d}_{+})}
+\lambda \|u\|_{L_{q,p}(\bR^{d}_{+})}
\leq N\sqrt{\lambda}\|g\|_{L_{q,p}(\bR^{d}_{+})}
+N\|f\|_{L_{q,p}(\bR^{d}_{+})},
$$
provided that $\lambda \ge \lambda_1$ and $u \in W^1_{q,p}(\bR^d_{+})$ satisfies
\begin{equation}	\label{eq090808_03}
\left\{
  \begin{array}{ll}
    \cL u-\lambda u=\Div g+f & \hbox{in $\bR^d_+$} \\
    a^{i1}u_{x^i}+a^1u=g_1 & \hbox{on $\partial \bR^d_+$}
  \end{array},
\right.
\end{equation}
Moreover, for any $\lambda > \lambda_1$ and $g,f\in L_{q,p}(\bR^{d}_{+})$, there exists a unique $u\in W^1_{q,p}(\bR^{d}_{+})$ satisfying \eqref{eq090808_03}.
\end{theorem}

Note that, similar to the homogeneous norm case, solutions of \eqref{eq090808_03} are understood in the weak sense, i.e. $u\in W^1_{q,p}(\bR_+^d)$ satisfies \eqref{eq090808_03} if we have
$$
\int_{\bR_+^d}\left(-a^{ij}u_{x^i}\phi_{x^j}-a^ju \phi_{x^j}+b^i u_{x^i}\phi+(c-\lambda)u\phi\right)\,dx=\int_{\bR_+^d}\left(-g_j\phi_{x^j}+f\phi\right)\,dx
$$
for any $\phi\in W^1_{q',p'}(\bR^d_+)$, where $q',p'$ satisfy $1/q+1/q'=1$ and $1/p+1/p'=1$.

\mysection{Mixed norm estimate of $u_{x^1}$}
                                                    \label{sec10}
In this section we set
$$
\cL_0 u = \left( a^{ij} u_{x^i} \right)_{x^j}
$$
and prove that the mixed norm of $u_{x^1}$ is controlled by
that of $g$ and $u_{x'}$ if $\cL_0 u = \Div g$ and $a^{ij}$ satisfy Assumption \ref{assumption20080424} ($\gamma$).

The first result of this section is an $L_p$-version of Lemma \ref{lemma081701}.
Since Theorem \ref{th081901}, more precisely, Theorem \ref{th081901} (iii) is now available, the proof of the lemma
is exactly the same as that of Lemma \ref{lemma081701}
with $p$ in place of $2$ and Theorem \ref{th081901}
in place of Theorem \ref{theorem08061901}.

\begin{lemma}							\label{lem082101}
Let $p \in (1,\infty)$, $\lambda > 0$, $r > 0$, $\kappa > 8 K \delta^{-1}$, and $a^{11} = a^{11}(x^1)$.
Assume that $u \in W_{p, \text{loc}}^1$ and
$$
\cL_0 u - \lambda u = \Div g + f,
$$
where $f$, $g \in L_{p,\text{loc}}$.
Then there exists a constant $N= N(d, p, \delta,K)$ such that
$$
\left( | \bar{a}^{11}\bar{u}_{x^1} - \left( \bar{a}^{11}\bar{u}_{x^1} \right)_{B_r} |^p \right)_{B_r}^{1/p}
\le N \big( \kappa^{-1} + \kappa^{d/p} \mu^{-1} \big) \left(|\bar{u}_{x^1}|^p\right)^{1/p}_{B_{\kappa r}}
$$
$$
+ N \kappa^{d/p} \left( |\bar{u}_{x'}|^p + \lambda^{p/2} |\bar{u}|^p + |\tilde{g}|^p
+ \lambda^{-p/2} |\tilde{f}|^p \right)^{1/p}_{B_{\kappa r}},
$$
for all $\mu \ge 1$,
where $\bar{a}^{ij}$, $\bar{u}$, $\tilde{f}$, and $\tilde{g}$  are those in \eqref{eq081401} and \eqref{eq081402}.
\end{lemma}

It is possible to derive a similar  but more complicated estimate
from the above lemma
if $a^{ij}$ are measurable in $x^1$ and BMO  in $x' \in \bR^{d-1}$.

\begin{theorem}							 \label{theorem082102}
Let $p \in (1,\infty)$.
If $u \in W_{p, \text{loc}}^1$ satisfies $\cL_0 u = \Div g$,
where $g \in L_{p,\text{loc}}$,
then for each $x_0 \in \bR^d$,
$\mu \ge 1$,
$\kappa > 16 K \delta^{-1}$,
and $r \in (0, \frac{R_0}{\sqrt{2}\kappa}]$,
there exist $\gamma_0=\gamma_0(d,p,\delta,K,\mu)$ and a measurable function $\bar{a}(x^1) = \bar{a}_{x_0, \mu,\kappa r}(x^1)$ such that $\delta \le \bar{a}(x^1) \le K$
and
$$
\left(| \bar{a} \bar{u}_{x^1} - \left(\bar{a} \bar{u}_{x^1}\right)_{B_r(x_0)} |^p \right)^{1/p}_{B_r(x_0)}
\le N_1 \kappa^{d/p} \left( |\bar{u}_{x'}|^p + R_0^{-p}|\bar{u}|^p + |\tilde{g}|^p\right)^{1/p}_{B_{\sqrt{2}\kappa r}(x_0)}
$$
$$
+ N \big( \kappa^{-1} + \kappa^{d/p} \mu^{-1} + N_1 \kappa^{d/p} \gamma^{1/(2p)} \big) \left(|\bar{u}_{x^1}|^p\right)^{1/p}_{B_{\sqrt{2}\kappa r}(x_0)},
$$
provided that $a^{ij}$ satisfy Assumption \ref{assumption20080424} ($\gamma$) and $\gamma\le \gamma_0$.
Here $N = N(d,p,\delta,K)$, independent of $\mu$, and
$N_1=N_1(d,p,\delta,K,\mu)$.
Recall that $\bar{u}$ and $\tilde{g}$  are those in \eqref{eq081401} and \eqref{eq081402}.
\end{theorem}

\begin{proof}
First we prove the case when $x_0$ is the origin.
By a scaling, it suffices to consider the case $R_0=1$.
For given $\mu \ge 1$, $\kappa > 16 K \delta^{-1}$, and $r \in (0, \frac{1}{\sqrt{2}\kappa}]$,
denote
$$
C^{\mu}_{r} = (-\mu^{-1} r, \mu^{-1} r) \times B'_r,
\quad
C_{r} = (-r, r) \times B'_r.
$$
Fix a $\lambda > \lambda_0$
and let $\gamma_0 \le \gamma$,
where $\lambda_0 = \lambda_0(d,p,\delta,K)$ and $\gamma=\gamma(d,p,\delta,K)$ are taken from Theorem \ref{th081901}.
By Theorem \ref{th081901}
there exists a unique solution $w \in W_p^1$ to the equation
$$
\cL_0 w - \lambda w = \Div (I_{C^{\mu}_{\kappa r}} g) + I_{C^{\mu}_{\kappa r}} f_{\lambda},
$$
where $f_{\lambda} =  - \lambda u$.
It then follows that
$$
\sqrt{\lambda} \| w \|_{L_p} + \| w_{x} \|_{L_p}
\le N \left(\| I_{C^{\mu}_{\kappa r}} g \|_{L_p} + \lambda^{-1/2} \| I_{C^{\mu}_{\kappa r}} f_{\lambda} \|_{L_p}\right),
$$
where $N = N(d,p,\delta,K)$ is
 independent of $\mu$.
The notations in \eqref{eq081401} and \eqref{eq081402} turn the above estimate into
$$
\sqrt{\lambda} \| \bar{w} \|_{L_p} + \mu \| \bar{w}_{x^1} \|_{L_p}
+ \| \bar{w}_{x'} \|_{L_p}
$$
$$
\le N \mu^{-1} \| I_{C_{\kappa r}} \tilde{g}_1 \|_{L_p}
+ N \sum_{j\ge2}\| I_{C_{\kappa r}} \tilde{g}_j \|_{L_p}
+ N \lambda^{1/2} \|I_{C_{\kappa r}} \bar{u} \|_{L_p},
$$
where $N = N(d,p,\delta,K)$.
This indicates that
$$
\left( |\bar w_x|^p \right)_{B_r}^{1/p}
\le N \kappa^{d/p} \big( |\tilde{g}|^p + \lambda^{p/2} |\bar{u}|^p \big)^{1/p}_{B_{\sqrt{2}\kappa r}},
$$
$$
\left( |\bar w_x|^p + \lambda^{p/2} |\bar w|^p \right)_{B_{\kappa r}}^{1/p}
\le N \big( |\tilde{g}|^p + \lambda^{p/2} |\bar{u}|^p \big)^{1/p}_{B_{\sqrt{2}\kappa r}}
$$
since $\mu \ge 1$
and $C_{\kappa r} \subset B_{\sqrt{2}\kappa r}$.

Observe that $v := u - w$ satisfies
$$
\cL_0 v - \lambda v = \Div \left( (1- I_{C^{\mu}_{\kappa r}}) g\right) + (1 - I_{C^{\mu}_{\kappa r}}) f_{\lambda}.
$$
Define $\fL_0$ to be the operator given by replacing the coefficients $a^{11}$ of $\cL_0$ with
$$
a(x^1) = \dashint_{B'_{\kappa r/2}} a^{11}(x^1,y') \, dy'.
$$
We also set $\bar a(x^1)=a(\mu^{-1}x^1)$. Then
$$
\fL_0 v - \lambda v = \Div \mathfrak{g} + \mathfrak{f},
$$
where
$$
\mathfrak{g}_1 = (a - a^{11})v_{x^1} + (1- I_{C^{\mu}_{\kappa r}}) g_1,
\quad
\mathfrak{g}_j = (1- I_{C^{\mu}_{\kappa r}}) g_j,
\quad
j \ge 2,
$$
$$
\mathfrak{f} = (1 - I_{C^{\mu}_{\kappa r}}) f_{\lambda}.
$$
Note that $\kappa/2 > 8 K \delta^{-1}$. By Lemma \ref{lem082101} applied to
the above equation, we get
$$
\left( | \bar{a}\bar{v}_{x^1} - \left( \bar{a}\bar{v}_{x^1} \right)_{B_r} |^p \right)_{B_r}^{1/p}
\le N \big( \kappa^{-1} + \kappa^{d/p} \mu^{-1} \big) \left(|\bar{v}_{x^1}|^p\right)^{1/p}_{B_{\kappa r/2}}
$$
$$
+ N \kappa^{d/p} \left( |\bar{v}_{x'}|^p + \lambda^{p/2} |\bar{v}|^p + |\tilde{\mathfrak{g}}|^p
+ \lambda^{-p/2} |\tilde{\mathfrak{f}}|^p \right)^{1/p}_{B_{\kappa r/2}}.
$$
Due to the indicator functions in front of $g$ and $f_{\lambda}$, we see that
$$
\left( |\tilde{\mathfrak{g}}|^p + \lambda^{-p/2} |\tilde{\mathfrak{f}}|^p \right)_{B_{\kappa r/2}}
= \dashint_{B_{\kappa r/2}} | \mu^2 (\bar{a} - \bar{a}^{11}) \bar{v}_{x^1} |^p \, dx
$$
$$
\le \mu^{2p} \left(\dashint_{B_{\kappa r/2}} |\bar{a} - \bar{a}^{11}|^{2p} \, dx \right)^{1/2}
\left( \dashint_{B_{\kappa r/2}} |\bar{v}_{x^1}|^{2p} \, dx \right)^{1/2}.
$$
To estimate the last term in the above inequality, note that, in $C_{\kappa r}^{\mu}$,
$$
\cL_0 v - \lambda v = 0.
$$
Thus $\bar{v}$ satisfies, in $C_{\kappa r} \supset B_{\kappa r}$,
$$
\left(\check{a}^{ij} \bar{v}_{x^i}\right)_{x^j}
- \mu^{-2} \lambda \bar{v} = 0,
$$
where
$$
\check{a}^{11} = \bar{a}^{11},
\quad
\check{a}^{1j} = \mu^{-1}\bar{a}^{1j},
\quad
\check{a}^{i1} = \mu^{-1}\bar{a}^{i1},
\quad
\check{a}^{ij} = \mu^{-2}\bar{a}^{ij},
\quad
i,j \ge 2.
$$
A calculation along with the fact $\mu \ge 1$ shows that
$$
\text{osc}_{x'}\left(\check{a}^{11},\Gamma_r(x)\right)
\le \mu \, \text{osc}_{x'}\left(a^{11},\Gamma_r(\mu^{-1}x^1,x')\right).
$$
We have similar inequalities for the other coefficients,
so we have $\check{a}_{R}^{\#}\le\mu a_{R}^{\#}$.
As to the boundedness and the uniform ellipticity constant of these coefficients,
we see that they are bounded by $K$ as $a^{ij}$,
but the ellipticity constant is $\mu^{-2} \delta$ instead of $\delta$.
Find $\gamma_0$ such that $\check{a}^{\#}_{1} \le \gamma$,
where $\gamma = \gamma(d, p, 2p, \mu^{-2}\delta, K)$ is taken from
Corollary \ref{cor082002}.
Then by Corollary \ref{cor082002} along with $\kappa r \le 1/\sqrt 2$
there exists a constant $N_1 = N_1( d, p, \delta, K, \mu)$
such that
$$
\left( \dashint_{B_{\kappa r/2}} |\bar{v}_{x^1}|^{2p} \, dx \right)^{1/2}
\le N_1 \left( \dashint_{B_{\kappa r}} |\bar{v}_{x}|^{p} + \mu^{-p}\lambda^{p/2} |\bar{v}|^p \, dx \right).
$$

On the other hand,
$$
\dashint_{B_{\kappa r/2}} |\bar{a} - \bar{a}^{11}|^{2p} \, dx
\le \mu(2K)^{2p-1}\gamma.
$$

To finish the proof of the case $x_0 = 0$, we combine all the inequalities above as in the proof of Lemma \ref{lemma081701}.
We also bear in mind that the fixed $\lambda$ is a constant depending only on $d$, $p$, $\delta$ and $K$.
For the general $B_{r}(x_0)$, $x_0 = (x_0^1, x_0')$, we use a translation $u(x^1, x') \to u(x^1+ \mu^{-1}x_0^1, x' + x_0')$,
which gives $\bar{u}(x) \to \bar{u}(x+x_0)$.
\end{proof}

Recall that, for $x = (x^1,x^2, \cdots, x^d) \in \bR^{d}$,
$\sbx_1$ represents the first $d_1$ coordinates of $x$
and $\sbx_2$ represents the remaining $d_2$ coordinates of $x$,
where $d_1, d_2 > 0$ and $d_1 + d_2 = d$.
Let
$$
B^{d_1}_r(\sbx_1) = \{ |\sbx_1 - \sby_1| < r : \sby_1 \in \bR^{d_1} \},
\quad
B^{d_2}_r(\sbx_2) = \{ |\sbx_2 - \sby_2| < r : \sby_2 \in \bR^{d_2} \}.
$$
As before, we set $B^{d_1}_r = B^{d_1}_r(0)$ and $B^{d_2}_r = B^{d_2}_r(0)$.
For a function $f$ defined on $\bR^{d}$, denote
$$
\| f(\cdot, \sbx_2)\|_{p,d_1}
= \left(\int_{\bR^{d_1}} | f(\sbx_1, \sbx_2) |^p \, d \sbx_1\right)^{1/p}.
$$
Note that $\| f(\cdot, \sbx_2)\|_{p,d_1}$ is a function of $\sbx_2$.

\begin{corollary}							 \label{cor081102}
Under the assumptions of Theorem \ref{theorem082102},
there exist constants $N = N(d,p,\delta,K)$
and $N_1 = N_1(d,p,\delta,K,\mu)$
such that
$$
\dashint_{B^{d_2}_r}\dashint_{B^{d_2}_r}
\left| \|\bar{u}_{x^1}(\cdot, \sbx_2)\|_{p,d_1} - \|\bar{u}_{x^1}(\cdot, \sby_2)\|_{p,d_1} \right|^p\, d \sbx_2 \, d \sby_2
$$
$$
\le N_1 \kappa^{d} \dashint_{B^{d_2}_{2\kappa r}}\|\bar{u}_{x'}(\cdot, \sbx_2)\|_{p,d_1}^p
+  R_0^{-p}\|\bar{u}(\cdot, \sbx_2)\|_{p,d_1}^p +  \| \tilde{g}(\cdot, \sbx_2)\|_{p,d_1}^p \, d \sbx_2
$$
$$
+ N \big( \kappa^{-p} + \kappa^{d} \mu^{-p} + N_1 \kappa^{d} \gamma^{1/2} \big)
\dashint_{B^{d_2}_{2\kappa r}} \|\bar{u}_{x^1}(\cdot, \sbx_2)\|_{p,d_1}^p \, d \sbx_2
$$
for all $\kappa > 16 K \delta^{-1}$ and $r \in (0, \frac{ R_0}{2\kappa}]$.
\end{corollary}

\begin{proof}
Fix $\kappa > 16 K \delta^{-1}$ and $r \in (0, \frac{R_0}{2\kappa}]$.
Note that
$$
\left| \|\bar{u}_{x^1}(\cdot, \sbx_2)\|_{p,d_1} - \|\bar{u}_{x^1}(\cdot, \sby_2)\|_{p,d_1} \right|^p
\le \|\bar{u}_{x^1}(\cdot, \sbx_2)- \bar{u}_{x^1}(\cdot, \sby_2)\|^p_{p,d_1}
$$
$$
= \dashint_{B^{d_1}_r} \int_{\bR^{d_1}} \left| \bar{u}_{x^1}(\sbz_1+\sbw_1,\sbx_2) - \bar{u}_{x^1}(\sbz_1+\sbw_1,\sby_2) \right|^p \, d \sbz_1 \, d \sbw_1
$$
$$
=\int_{\bR^{d_1}} \dashint_{B^{d_1}_r(\sbz_1)} \left| \bar{u}_{x^1}(\sbw_1,\sbx_2) - \bar{u}_{x^1}(\sbw_1,\sby_2) \right|^p \, d \sbw_1 \, d \sbz_1.
$$
We
use Theorem \ref{theorem082102} to find $\bar{a}_{\sbz_1} = \bar{a}_{(\sbz_1,0), \mu, \sqrt{2}\kappa r}$ corresponding to $B_{\sqrt{2}r}(\sbz_1, 0)$.
Since $\delta \le  \bar{a}_{\sbz_1}\le K$, the last term above is not greater than $\delta^{-p}$ times
$$
\int_{\bR^{d_1}} \dashint_{B^{d_1}_r(\sbz_1)} \left| \bar{a}_{\sbz_1}\bar{u}_{x^1}(\sbw_1,\sbx_2) - \bar{a}_{\sbz_1}\bar{u}_{x^1}(\sbw_1,\sby_2) \right|^p \, d \sbw_1 \, d \sbz_1.
$$
Thus the left-hand side of the inequality in the corollary, denoted by $I$, satisfies
$$
I
\le N \int_{\bR^{d_1}} \dashint_{B^{d_2}_r} \dashint_{B^{d_2}_r}
\dashint_{B^{d_1}_r(\sbz_1)} \left| \bar{a}_{\sbz_1}\bar{u}_{x^1}(\sbw_1,\sbx_2) - \bar{a}_{\sbz_1}\bar{u}_{x^1}(\sbw_1,\sby_2) \right|^p d \sbw_1 \, d \sbx_2 \, d \sby_2 \, d \sbz_1.
$$
Observe that
$$
| \bar{a}_{\sbz_1}\bar{u}_{x^1}(\sbw_1,\sbx_2) - \bar{a}_{\sbz_1}\bar{u}_{x^1}(\sbw_1,\sby_2) |^p
$$
$$
\le 2^{p} | \bar{a}_{\sbz_1}\bar{u}_{x^1}(\sbw_1, \sbx_2) - \left( \bar{a}_{\sbz_1}\bar{u}_{x^1} \right)_{B_{\sqrt{2}r}(\sbz_1, 0)} |^p
+ 2^{p} | \bar{a}_{\sbz_1}\bar{u}_{x^1}(\sbw_1,\sby_2) - \left( \bar{a}_{\sbz_1}\bar{u}_{x^1} \right)_{B_{\sqrt{2}r}(\sbz_1, 0)} |^p
$$
and
$$
B^{d_1}_r(\sbz_1) \times B^{d_2}_r
\subset B_{\sqrt{2}r}(\sbz_1, 0).
$$
Hence
$$
I \le N
\int_{\bR^{d_1}} \dashint_{B_{\sqrt{2}r}(\sbz_1,0)} | \bar{a}_{\sbz_1}\bar{u}_{x^1}(x) - \left( \bar{a}_{\sbz_1}\bar{u}_{x^1} \right)_{B_{\sqrt{2}r}(\sbz_1,0)} |^p \, d x \, d \sbz_1,
$$
where $N$ depends only on $d$ and $p$.
Then by Theorem \ref{theorem082102} we have
$$
I \le N_1 \kappa^{d} \int_{\bR^{d_1}} \left( |\bar{u}_{x'}|^p + R_0^{-p}|\bar{u}|^p + |\tilde{g}|^p\right)_{B_{2 \kappa r}(\sbz_1,0)} \, d \sbz_1
$$
$$
+ N \big( \kappa^{-p} + \kappa^{d} \mu^{-p} + N_1 \kappa^{d} \gamma^{1/2}\big)
\int_{\bR^{d_1}} \left(|\bar{u}_{x^1}|^p\right)_{B_{2\kappa r}(\sbz_1,0)} \, d \sbz_1,
$$
where $N$ is independent of $\mu$.

The same process as at the beginning of the proof yields,
for example,
$$
\int_{\bR^{d_1}} \left( | \tilde{g} |^p \right)_{B_{2\kappa r}(\sbz_1,0)} \, d \sbz_1
= \int_{\bR^{d_1}} \dashint_{B_{2\kappa r}} |\tilde{g}(\sbz_1 + \sbw_1,\sbx_2)|^p \, d \sbw_1 \, d \sbx_2 \, d \sbz_1
$$
$$
\le N(d) \int_{\bR^{d_1}} \dashint_{B^{d_1}_{2\kappa r}} \dashint_{B^{d_2}_{2\kappa r}} |\tilde{g}(\sbz_1 + \sbw_1,\sbx_2)|^p \, d \sbx_2 \, d \sbw_1 \, d \sbz_1
$$
$$
= N \dashint_{B^{d_2}_{2\kappa r}} \int_{\bR^{d_1}} |\tilde{g}(\sbz_1,\sbx_2)|^p \, d \sbz_1 \, d \sbx_2
= N \dashint_{B^{d_2}_{2\kappa r}} \| \tilde{g}(\cdot, \sbx_2)\|_{p,d_1}^p \, d \sbx_2.
$$
Therefore, we obtain the inequality in the corollary.
\end{proof}

If $g$ is a function defined on $\bR^{d_2}$, naturally its maximal and sharp functions are
$$
M g (\sbx_2) = \sup_{r>0} \dashint_{B^{d_2}_r(\sbx_1)} |g(\sby_2)| \, d\sby_2,
$$
$$
g^{\#}(\sbx_2) = \sup_{r>0} \dashint_{B^{d_2}_r(\sbx_2)} |g(\sby_2) -
(g)_{B^{d_2}_r(\sbx_2)}| \, d\sby_2.
$$

We now come to the main result of this section.

\begin{theorem}							 \label{theorem082301}
Let $1 < p < q < \infty$. Then there exists a constant $\gamma=\gamma(d,p,q,\delta,K)$
such that, under Assumption \ref{assumption20080424} ($\gamma$),
the following holds true. There exist constants $N$ and $R_3\in (0,1]$, depending only on
$d_1$, $d_2$, $p$, $q$,
$\delta$ and $K$, such that,
for any $u \in C_0^{\infty}$
satisfying $\cL_0 u = \Div g$, where $g \in L_{q,p}$,
$$
\|u_{x^1}\|_{L_{q,p}} \le N \left( \|u_{x'}\|_{L_{q,p}} + R_0^{-1}\| u \|_{L_{q,p}} + \| g \|_{L_{q,p}} \right),
$$
provided that $u(\sbx_1, \sbx_2) = 0$ for $\sbx_2 \notin  B_{R^2}^{d_2} $, $R\in (0, R_3R_0]$.
\end{theorem}

\begin{proof}
Again we may assume $R_0=1$. Fix $\mu \ge 1$ and $\kappa \ge 16 K \delta^{-1}$, which are to be chosen below.
Let $\gamma \le \gamma_0$, where
$\gamma_0=\gamma_0(d,p,\delta,K,\mu)$ is taken from Theorem \ref{theorem082102}
and set
$$
\mathfrak{u}(\sbx_2) = \| \bar{u}_{x^1}(\cdot, \sbx_2) \|_{p,d_1},
$$
$$
\mathfrak{f}(\sbx_2) = \|\bar{u}_{x'}(\cdot, \sbx_2)\|_{p,d_1}
+  \|\bar{u}(\cdot, \sbx_2)\|_{p,d_1} +  \| \tilde{g}(\cdot, \sbx_2)\|_{p,d_1},
$$
where $\bar{u}$ and $\tilde{g}$ are defined as in \eqref{eq081401}
and \eqref{eq081402}.
If $r \le R/(2\kappa)$,
from Corollary \ref{cor081102} along with an appropriate translation as well as the H\"{o}lder's inequality
it follows that
$$
\dashint_{B^{d_2}_r(\bar{\sbx}_2)}
| \mathfrak{u} - \left(\mathfrak{u}\right)_{B^{d_2}_r(\bar{\sbx}_2)} | \, d \sbx_2
$$
$$
\le N_1 \kappa^{d/p} \left( \mathfrak{f}^p \right)^{1/p}_{B^{d_2}_{2\kappa r}(\bar{\sbx}_2)}
+ N \big( \kappa^{-1} + \kappa^{d/p} \mu^{-1} + N_1 \kappa^{d/p} \gamma^{1/(2p)} \big)
\left(\mathfrak{u}^p\right)^{1/p}_{B^{d_2}_{2\kappa r}(\bar{\sbx}_2)}
$$
for all $\bar{\sbx}_2 \in \bR^{d_2}$,
where $N$ is independent of $\mu$.
If  $r >  R /(2\kappa)$,
since $\mathfrak{u}$ has a compact support in $B_{ R^2 }^{d_2}$,
by the H\"{o}lder's inequality
$$
\dashint_{B^{d_2}_r(\bar{\sbx}_2)}
| \mathfrak{u} - \left(\mathfrak{u}\right)_{B^{d_2}_r(\bar{\sbx}_2)} | \, d \sbx_2
\le N \left(\dashint_{B^{d_2}_r(\bar{\sbx}_2)}
I_{B_{R^2}^{d_2}} \, \sbx_2 \right)^{1-1/p}
\left(\mathfrak{u}^p \right)_{B^{d_2}_r(\bar{\sbx}_2)}^{1/p}
$$
$$
\le N (R^2/r)^{d_2(1-1/p)} \left(\mathfrak{u}^p \right)_{B^{d_2}_r(\bar{\sbx}_2)}^{1/p}
\le N (\kappa R)^{d_2(1-1/p)} \left(\mathfrak{u}^p \right)_{B^{d_2}_r(\bar{\sbx}_2)}^{1/p}.
$$
Therefore, by the above two sets of inequalities as well as the fact that,
for example,
$\left(\mathfrak{f}^p\right)_{B^{d_2}_{2\kappa r}(\bar{\sbx}_2)}
\le M \mathfrak{f}^p (\bar{\sbx}_2)$,
we obtain
$$
\dashint_{B^{d_2}_r(\bar{\sbx}_2)}
| \mathfrak{u} - \left(\mathfrak{u}\right)_{B^{d_2}_r(\bar{\sbx}_2)} | \, d \sbx_2
\le N_1 \kappa^{d/p} \left(M \mathfrak{f}^p (\bar{\sbx}_2)\right)^{1/p}
$$
$$
+ N \big( \kappa^{-1} + \kappa^{d/p} \mu^{-1} + N_1 \kappa^{d/p} \gamma^{1/(2p)} + (\kappa R)^{d_2(1-1/p)} \big)
\left(M \mathfrak{u}^p (\bar{\sbx}_2) \right)^{1/p}
$$
for all $r > 0$ and $\bar{\sbx}_2 \in \bR^{d_2}$,
where $N$ is independent of $\mu$.
This implies the pointwise estimate that the sharp function $\mathfrak{u}^{\#}$ is bounded by the right-hand side of the inequality.
Then using the maximal function theorem and the Fefferman-Stein theorem, we get
(note that $q > p$)
$$
\| \mathfrak{u} \|_{L_q(\bR^{d_2})}
\le N_1 \kappa^{d/p} \| \mathfrak{f} \|_{L_q(\bR^{d_2})}
$$
$$
+ N \big( \kappa^{-1} + \kappa^{d/p} \mu^{-1} + N_1 \kappa^{d/p} \gamma^{1/(2p)} + (\kappa R)^{d_2(1-1/p)} \big)\| \mathfrak{u} \|_{L_q(\bR^{d_2})}.
$$
Bearing in mind that $N$ is independent of $\mu$, we choose first a sufficiently big $\kappa$, then a sufficiently big $\mu$, and finally sufficiently small $\gamma$ and $R_3$ so that
$$
N \big( \kappa^{-1} + \kappa^{d/p} \mu^{-1} + N_1 \kappa^{d/p}
\gamma^{1/(2p)} + (\kappa  R)^{d_2(1-1/p)} \big)
\le 1/2
$$
for all $R \le R_3$.
It then follows that
$$
\| \bar{u}_{x^1} \|_{L_{q,p}}
\le N \left(\| \bar{u}_{x'} \|_{L_{q,p}}
+   \| \bar{u} \|_{L_{q,p}}
+ \| \tilde{g} \|_{L_{q,p}}\right),
$$
where $N = N(d_1, d_2, p, q, \delta, K)$.
To finish the proof, we turn the above inequality into an inequality in terms of $u$ and $g$.
~\end{proof}

\mysection{Mixed norms}							 \label{mixednormsec}

Finally, in this section we prove Theorem \ref{thm6.5}.
First we present an $L_p$-version of Theorem \ref{thm2.05}.
Now that we have proved Theorem \ref{th081901} (iii),
which is an $L_p$-version of Theorem \ref{theorem08061901}
if $a^{ij}$ are measurable functions of $x^1 \in \bR$ only,
the following theorem is proved in the same manner as Theorem \ref{thm2.05}
using Theorem \ref{th081901} and Corollary \ref{cor080702}.

\begin{theorem}							\label{th082001}
Let $p \in (1,\infty)$, $\lambda > 0$, $\kappa\ge 4$, $r>0$,
and $a^i = b^i = c = 0$.
Assume that $a^{ij} = a^{ij}(x^1)$
and  $u\in W^1_{p,\text{loc}}$
satisfies $\cL u - \lambda u =\Div g + f$ in $B_{\kappa r}$,
where $f$, $g\in L_{p,\text{loc}}$.
Then there exist positive constants $N$ and $\alpha$, depending only on $d$, $p$, $\delta$, and $K$, such that
$$
\left(|u_{x'}-(u_{x'})_{B_r}|^p\right)_{B_r}\leq N\kappa^{-p\alpha}
\left(|u_{x}|^p + \lambda^{p/2} |u|^p\right)_{B_{\kappa r}}+N\kappa^{d}\left(|g|^p + \lambda^{-p/2}|f|^p\right)_{B_{\kappa r}}.
$$
\end{theorem}

Based on Corollary \ref{cor082002} and Theorem \ref{th081901}, we prove an estimate of the $L_p$-oscillations of $u_{x'}$ as follows.

\begin{theorem}							\label{th082301}
Let $p \in (1,\infty)$ and $a^j = b^i = c = 0$.
Assume that $u \in W_{p,\text{loc}}^1$
satisfies $\cL u = \Div g$, where $g \in L_{p,\text{loc}}$.
 Then there exists a constant
$\gamma_0=\gamma_0(d,p,q,\delta,K)$
such that, under Assumption \ref{assumption20080424} ($\gamma$),
$\gamma \le \gamma_0$,
the following holds true.
There exist positive constants $N = N(d,p,\delta,K)$
and $\alpha = \alpha(d,p,\delta,K)$ such that
$$
\left( | u_{x'} - \left(u_{x'}\right)_{B_r} |^p \right)^{1/p}_{B_r}
\le N \left( \kappa^{-\alpha} + \kappa^{d/p}\gamma^{1/(2p)} \right) \left( |u_x|^p \right)_{B_{\kappa r}}^{1/p}
$$
$$
+ N \kappa^{d/p} \left( |g|^p +  R_0^{-p}|u|^p \right)_{B_{\kappa r}}^{1/p}
$$
for all $\kappa \ge 8$, $r \in (0, \frac{R_0}{\sqrt{2}\kappa}]$.
\end{theorem}

\begin{proof}
By a scaling, we may again assume $R_0=1$. Fix a $\lambda>\lambda_0$ and let $\gamma_0 \le \gamma$,
where $\lambda_0 = \lambda_0(d,p,\delta,K)$ and $\gamma = \gamma(d,p,\delta,K)$ are from Theorem \ref{th081901}.
Then there exists $w \in W_p^1$
such that
$$
\cL w - \lambda w = \Div( I_{B_{\kappa r}} g ) - \lambda I_{B_{\kappa r}} u,
$$
$$
\sqrt{\lambda} \| w \|_{L_p} + \| w_x \|_{L_p}
\le N\| I_{B_{\kappa r}} g \|_{L_p} + N \sqrt{\lambda} \| I_{B_{\kappa r}} u \|_{L_p}.
$$
As before, this shows that
$$
\left(|w_x|^p\right)_{B_r}
\le N \kappa^{d} \left( |g|^p + \lambda^{p/2} |u|^p \right)_{B_{\kappa r}},
$$
$$
\left(|w_x|^p\right)_{B_{\kappa r}}
+ \lambda  \left(|w|^p\right)_{B_{\kappa r}}
\le N \left(|g|^p + \lambda^{p/2}|u|^p\right)_{B_{\kappa r}}.
$$

By setting $v := u - w$ we observe that $v \in W_{p,\text{loc}}^1$,
$$
\cL v - \lambda v = \Div \left( (1 - I_{B_{\kappa r}}) g \right) - \lambda (1 - I_{B_{\kappa r}}) u,
$$
and $\cL v - \lambda v = 0$ in $B_{\kappa r}$.

Let
$$
\sba^{ij}(x^1) = \dashint_{B'_{\kappa r/2}} a^{ij}(x^1,y') \, dy',\quad
\bar{\cL} \varphi  = (\sba^{ij} \varphi_{x^i})_{x^j}.
$$
Then due to the fact that $\cL v - \lambda v = 0$ in $B_{\kappa r}$,
$$
\bar{\cL} v  - \lambda v = \left( (\sba^{ij} - a^{ij}) v_{x^i} \right)_{x^j}
$$
in $B_{\kappa r}$.
Since $\kappa/2 \ge 4$, by Theorem \ref{th082001} applied to the operator $\bar{\cL}$
$$
\left(|v_{x'}-(v_{x'})_{B_r}|^p\right)_{B_r}
\leq N\kappa^{-p\alpha}
\left(|v_{x}|^p + \lambda^{p/2} |v|^p\right)_{B_{\kappa r/2}}+N\kappa^{d}\left(|\bar{g}|^p\right)_{B_{\kappa r/2}},
$$
where $\bar{g}_j = (\bar{a}^{ij} - a^{ij}) v_{x^i}$.
Note that
$$
\left(|\bar{g}|^p\right)_{B_{\kappa r/2}}
\le \left(|\sba^{ij} - a^{ij}|^{2p}\right)_{B_{\kappa r/2}}^{1/2}
\left(|v_x|^{2p}\right)_{B_{\kappa r/2}}^{1/2}
=: I_1^{1/2} I_2^{1/2},
$$
where $I_1 \le N a_{\kappa r/2}^{\#}$.
Under the assumption that
$\gamma_0 \le \gamma$,
where $\gamma=\gamma(d,p,2p,\delta,k)$ in Corollary \ref{cor082002}, we have
by Corollary \ref{cor082002}
applied to the fact that $\cL v - \lambda v = 0$ in $B_{\kappa r}$
$$
I_2 \le N \left(|v_x|^{p} + \lambda^{p/2}|v|^p\right)_{B_{\kappa r}}^2.
$$
Here we also used the fact $\kappa r \le 1/\sqrt{2}$.
Now to finish the proof we proceed as in the proof of Theorem \ref{th080601}.
\end{proof}

Theorem \ref{th082301} along with the argument in the proof of Corollary \ref{cor081102} yields 

\begin{corollary}							 \label{cor082301}
Under the assumptions of Theorem \ref{th082301},
there exists a constant $N = N(d,p,\delta,K)$
such that
$$
\dashint_{B^{d_2}_r}\dashint_{B^{d_2}_r}
\left| \|u_{x'}(\cdot, \sbx_2)\|_{p,d_1} - \|u_{x'}(\cdot, \sby_2,)\|_{p,d_1} \right|^p\, d \sbx_2 \, d \sby_2
$$
$$
\le N \left( \kappa^{-p\alpha} + \kappa^{d} \gamma^{1/2} \right) \dashint_{B^{d_2}_{2\kappa r}}\|u_x(\cdot, \sbx_2)\|_{p,d_1}^p \, d \sbx_2
$$
$$
+ N \kappa^{d/p} \dashint_{B^{d_2}_{2\kappa r}}\|g(\cdot, \sbx_2)\|_{p,d_1}^p
+R_0^{-p}\|u(\cdot, \sbx_2)\|_{p,d_1}^p \, d \sbx_2.
$$
for all $\kappa \ge 8$, $r \in (0, \frac{R_0}{\sqrt{2}\kappa}]$.
\end{corollary}

By adopting the same strategy as in the proof of Theorem \ref{theorem082301}
as well as using the argument in the last part of the proof of Lemma \ref{lem3.52},
we obtain the following lemma from Corollary \ref{cor082301}.

\begin{lemma}							\label{lemma082302}
Let $1 < p < q < \infty$,
$a^i = b^i = c = 0$.
Then there exists a constant $\gamma=\gamma(d,p,q,\delta,K)$
such that, under Assumption \ref{assumption20080424} ($\gamma$),
the following holds true.
There exist constants $N$ and $ R_3\in (0,1]$, depending only on
$d_1$, $d_2$, $p$, $q$,
$\delta$ and $K$,
such that,
for $u \in C_0^{\infty}$
satisfying $\cL u = \Div g$, where $g \in L_{q,p}$,
$$
\|u_{x}\|_{L_{q,p}} \le N \left( \| g \|_{L_{q,p}} +   R_0^{-p}\| u \|_{L_{q,p}} \right),
$$
provided that $u(\sbx_1, \sbx_2) = 0$ for $\sbx_2 \notin B_{(R_3R_0)^2}^{d_2}$.
\end{lemma}

By modifying the proof of Lemma 5.5 in \cite{Krylov_2005}
and using Lemma \ref{lemma082302} above,
we prove the following lemma.

\begin{lemma}							\label{lem082601}
Let $1 < p < q < \infty$,
$f$, $g = (g_1, \cdots, g_d) \in L_{q,p}$, $ u \in C_0^{\infty}$, and
$$
\cL u - \lambda u = \Div g + f.
$$
Then there exists a constant $\gamma=\gamma(d,p,q,\delta,K)$
such that, under Assumption \ref{assumption20080424} ($\gamma$),
the following holds true.
There exist constants $R_3\in (0,1]$, $\lambda_1$ and $N$ depending only on
$d_1$, $d_2$, $p$, $q$,
$\delta$ and $K$,
such that
$$
\lambda \| u \|_{L_{q,p}}
+ \sqrt{\lambda}\|u_{x}\|_{L_{q,p}} \le N \left( \sqrt{\lambda}\| g \|_{L_{q,p}} + \| f \|_{L_{q,p}} \right),
$$
provided that $u(\sbx_1, \sbx_2) = 0$ for $\sbx_2 \notin B_{(R_3R_0)^2}^{d_2}$
and $\lambda > \lambda_1$.
\end{lemma}

\begin{proof}[Proof of Theorem \ref{thm6.5}]
If $p = q$, the theorem is a special case of Theorem \ref{th081901}.
The case $q < p$ is proved by the duality argument,
so we assume that $q > p$.
In this case,
it suffices to prove the estimate in the theorem for $u \in C_0^{\infty}$, which, by Lemma \ref{lem082601}, holds true
for $u$ with a small compact support with respect to $\sbx_2 \in \bR^{d_2}$.
Then we finish the proof by using a partition of unity (see the proofs of Theorem 5.7 in \cite{Krylov_2005}
or Lemma 3.4 in \cite{Krylov_2007_mixed_VMO}).
\end{proof}


\section*{Acknowledgement}

The authors are sincerely grateful to Nicolai V. Krylov and the referees for very helpful suggestions and comments.

\bibliographystyle{plain}

\def\cprime{$'$}\def\cprime{$'$} \def\cprime{$'$} \def\cprime{$'$}
  \def\cprime{$'$} \def\cprime{$'$}

\end{document}